\documentclass[12pt]{article}

\usepackage{amssymb}
\usepackage{amsthm}
\usepackage{amsmath}
\usepackage{graphicx}
\usepackage{fullpage}
\usepackage{color}
\usepackage{enumerate}
 \numberwithin{equation}{section}
\usepackage{booktabs}
\allowdisplaybreaks

\usepackage[pdftex]{hyperref}
 \usepackage{mathtools}
 \mathtoolsset{showonlyrefs}
\usepackage[cal=boondox]{mathalfa}

\theoremstyle{plain}
\newtheorem{thm}{Theorem}[section]
\newtheorem{cor}[thm]{Corollary}

\newtheorem{lem}[thm]{Lemma}
\newtheorem{prop}[thm]{Proposition}

\theoremstyle{definition}
\newtheorem{defn}[thm]{Definition}

\theoremstyle{remark}

\newtheorem{rem}[thm]{Remark}

\newcommand{\N}{\mathbb{N}}
\newcommand{\M}{\mathbb{M}}
\newcommand{\R}{\mathbb{R}}

\newcommand{\E}{\mathbb{E}}

\newcommand{\Div}{\text{div}}
\newcommand{\bp}{\begin{proof}[\ensuremath{\mathbf{Proof}}]}
\newcommand{\ep}{\end{proof}}

\newcommand{\ba}{{\bf a}}
\newcommand{\bu}{{\bf u}}
\newcommand{\bg}{{\bf g}}

\newcommand{\bh}{{\bf h}}
\newcommand{\bw}{{\bf w}}

\newcommand{\bzero}{{\bf 0}}

\newcommand{\Mnd}{\M^{n\times d}}
\newcommand{\Md}{\M^{d\times d}}
\newcommand{\be}{\begin{equation}}
\newcommand{\ee}{\end{equation}}

\newcommand{\id}{\text{id}_{\R}}

\begin{document}

\title{Newton's second law with a semiconvex potential}

\author{Ryan Hynd\footnote{Department of Mathematics, University of Pennsylvania.  Partially supported by NSF grant DMS-1554130.}}

\maketitle

\begin{abstract}
We make the elementary observation that the differential equation associated with Newton's second law $m\ddot\gamma(t)=-D V(\gamma(t))$ always has a solution for given initial conditions provided that the potential energy $V$ is semiconvex. That is, if $-D V$ satisfies a one-sided Lipschitz condition. We will then build upon this idea to verify the existence of solutions for the Jeans-Vlasov equation, the pressureless Euler equations in one spatial dimension, and the equations of elastodynamics under appropriate semiconvexity assumptions. 
\end{abstract}

\tableofcontents

\section{Introduction}
Newton's second law asserts that the trajectory $\gamma:[0,\infty)\rightarrow\R^d$ of a particle with mass $m>0$ satisfies the ordinary differential equation
\be\label{Newton}
m\ddot\gamma(t)=-D V(\gamma(t)),\quad t>0.
\ee 
Here $V: \R^d\rightarrow \R$ is the potential energy of the particle in the sense that $-D V(x)$ is the force acting on the particle when it is located at position $x$.  It is well known that 
equation \eqref{Newton} has a unique solution for given initial conditions 
\be\label{NewtonInit}
\gamma(0)=x_0\quad \text{and}\quad \dot\gamma(0)=v_0
\ee
provided that $-D V:\R^d\rightarrow \R^d$ is Lipschitz continuous.

\par It is not hard to show that the existence of solutions to \eqref{Newton} which satisfies \eqref{NewtonInit} still holds provided $-D V$ satisfies the one sided Lipschitz condition 
\be\label{SemiV}
(D V(x)-D V(y))\cdot (x-y)\ge -L |x-y|^2\quad (x,y\in \R^d)
\ee
for some $L\ge0$. We note $V$ satisfies \eqref{SemiV} if and only if $x\mapsto V(x)+(L/2)|x|^2$ is convex. Therefore, any such $V$ is called 
{\it semiconvex}.  The prototypical potentials we have in mind are convex for large values of $|x|$ as displayed in Figure \ref{PotentialFigRef}.

\begin{figure}[h]
\centering
 \includegraphics[width=.65\textwidth]{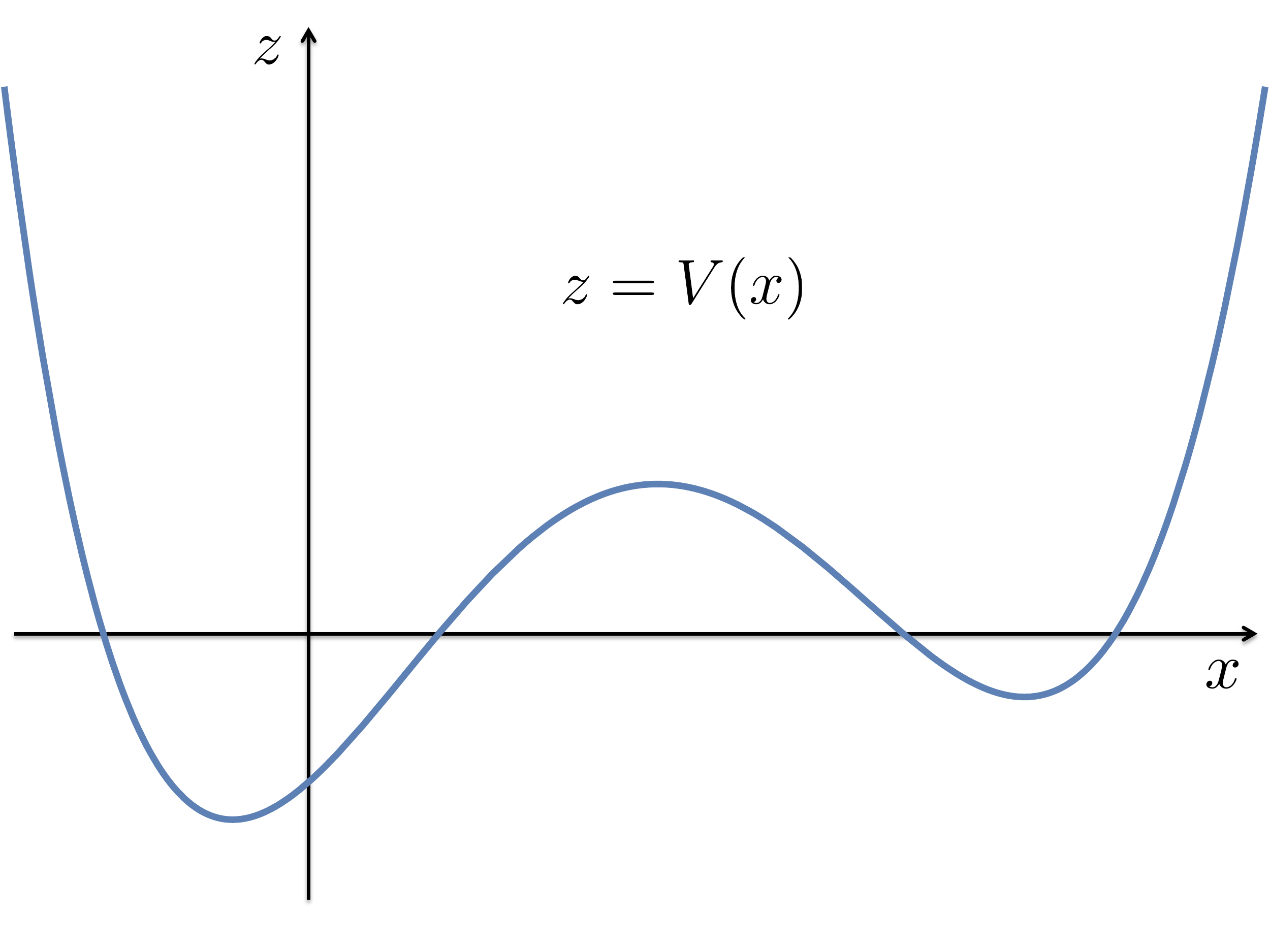}
 \caption{Schematic of a semiconvex potential $V$. }\label{PotentialFigRef}
\end{figure}

\par Our primary interest in considering Newton's second law with a semiconvex potential is to study equations arising in more complex physical models with similar underlying structure.  The first of these models involves the Jeans-Vlasov equation
\be
\partial_t f+v\cdot D_xf-D W*\rho \cdot D_v f=0.
\ee
This is a partial differential equation (PDE) for a time dependent mass distribution $f$ of particles in position and velocity space $(x,v)\in \R^d\times \R^d$ which interact pairwise via the potential $W:\R^d\rightarrow \R$; here $\rho$ denotes the spatial distribution of particles.  We will show that if $W$ is semiconvex, there is a weak solution of this PDE for each given initial mass distribution.    This is the content of section \ref{JVsect} below.

\par The second model we will consider is the pressureless Euler system in one spatial dimension
\be
\begin{cases}
\hspace{.27in}\partial_t\rho+\partial_x(\rho v)=0\\
\partial_t(\rho v)+\partial_x(\rho v^2)=-\rho(W'*\rho).
\end{cases}
\ee
These equations hold in $\R\times(0,\infty)$ and the unknowns are the spatial distribution of particles $\rho$ and a corresponding local velocity field $v$. The particles in question are constrained to move on the real line, interact pairwise via the potential energy $W:\R\rightarrow \R$ and are ``sticky" in the sense that they undergo perfectly inelastic collisions when they collide.  In section \ref{PEsect}, we will verify the existence of a weak solution pair $\rho$ and $v$ for given initial conditions which satisfies the entropy inequality
$$
(v(x,t)-v(y,t))(x-y)\le \frac{\sqrt{L}}{\tanh(\sqrt{L}t)}(x-y)^2
$$
for $x,y$ belonging to the support of $\rho$. Here $L>0$ is chosen so that $W(x)+(L/2)x^2$ is convex. As a result, the semiconvexity of $W$ will play a crucial role in our study.

\par In the final section of this paper, we will consider the equations of motion in the theory of elastodynamics
\be
\bu_{tt}=\text{div}DF(D\bu).
\ee
The unknown is a mapping $\bu:U\times[0,\infty)\rightarrow \R^d$, with $U\subset \R^d$, which encodes the displacement of a $d$ dimensional elastic material.  In particular, this is a system of $d$ coupled PDE for the components of $\bu$.  The gradient $D\bu$ is $d\times d$ matrix valued, and we will assume that $F$ is a semiconvex function on the space of real $d\times d$ matrices.  Under this hypothesis, we use Galerkin's method to approximate measure valued solutions of this system. We also show how to adapt these methods to verify the existence of weak solutions of the perturbed system
\be
\bu_{tt}=\text{div}DF(D\bu)+\eta\Delta \bu_t
\ee
for each $\eta>0$.

\par We acknowledge that the results verified below are not all new.  The existence of weak solutions to Jeans-Vlasov equation with a semiconvex interaction potential was essentially obtained by Ambrosio and Gangbo \cite{MR2361303};  in a recent paper which motivated this study \cite{Hynd}, we verified the existence of solutions of the pressureless Euler system in one spatial dimension with a semiconvex interaction potential; Demoulini \cite{MR1437192} established the existence of measure valued solutions to the equations of elastodynamics with nonconvex stored energy; and the existence of weak solutions of the corresponding perturbed system was verified by Dolzmann and Friesecke \cite{MR1434040}.  Nevertheless, we contend that our approach to verifying existence is unifying.  In particular, we present a general method to address the existence of weak and measure valued solutions of hyperbolic evolution equations when the underlying nonlinearity is appropriately semiconvex.

\section{Preliminaries}
We will be begin our study by showing the ODE associated with Newton's second law always has a solution for given initial conditions provided the potential energy is semiconvex. With some particular applications in mind, we will study an equation slightly more general than \eqref{Newton}. We will also review the weak convergence of Borel probability measures on a metric space.  These observations will be useful when we need to pass to the limit within various approximations built in later sections of this work.

\subsection{Newton systems}\label{ODEsect}
We fix $N\in \N$ and set 
$$
(\R^d)^N:=\prod^N_{i=1}\R^d=\R^d\times\dots\times \R^d.
$$
We represent points $x\in (\R^d)^N$ as $x=(x_1,\dots, x_N)$, where each $x_i\in \R^d$. We also write
$$
x\cdot y:=\sum^N_{i=1}x_i\cdot y_i
$$
and 
$$
|x|:=\left(\sum^N_{i=1}|x_i|^2\right)^{1/2}
$$
for $x,y\in (\R^d)^N$. Note that these definitions coincide with the standard dot product and norms on $(\R^d)^N$.  We will say that $V: (\R^d)^N\rightarrow \R$ is semiconvex provided 
\be\label{VwithLSemi}
(\R^d)^N\ni x\mapsto V(x)+\frac{C}{2}|x|^2
\ee
is convex for some $C\ge 0$.

\par  Let us now consider the {\it Newton system}
\begin{equation}\label{NewtonSystem}
\begin{cases}
m_i\ddot\gamma_i(t)=-D_{x_i}V(\gamma_1(t),\dots, \gamma_N(t)), \quad t>0\\
\hspace{.16in}\gamma_i(0)=x_i,\\
\hspace{.16in}\dot\gamma_i(0)=v_i.
\end{cases}
\end{equation}
Here $m_i>0$ and $(x_i,v_i)\in \R^d\times \R^d$ are given for $i=1,\dots, N$, and we seek a solution $\gamma_1,\dots, \gamma_N:[0,\infty)\rightarrow\R^d$.  It is easy to check that the conservation of energy holds for any solution:
\be\label{ElemConsEnergy}
\sum^N_{i=1}\frac{m_i}{2}|\dot\gamma_i(t)|^2+V(\gamma_1(t),\dots, \gamma_N(t))=
\sum^N_{i=1}\frac{m_i}{2}|v_i|^2+V(x_1,\dots, x_N)
\ee
for $t\ge 0$.  When $V$ is semiconvex, we can use this identity to prove the following assertion. 

\begin{prop}\label{ODENewtonProp}
Suppose that $V: (\R^d)^N\rightarrow \R$ is continuously differentiable
and semiconvex. Then \eqref{NewtonSystem} has a solution $\gamma_1,\dots, \gamma_N: C^2([0,\infty);\R^d)$.
\end{prop}

\begin{proof}
By Peano's existence theorem, there is a solution $\gamma_1,\dots, \gamma_N:[0,T)\rightarrow\R^d$ of 
\begin{equation}\label{NewtonSystemT}
\begin{cases}
m_i\ddot\gamma_i(t)=-D_{x_i}V(\gamma_1(t),\dots, \gamma_N(t)), \quad t\in (0,T)\\
\hspace{.16in}\gamma_i(0)=x_i\\
\hspace{.16in}\dot\gamma_i(0)=v_i
\end{cases}
\end{equation}
for some $T\in (0,\infty]$. We may also assume that $[0,T)$ is the maximal interval of existence for this solution. In particular, if $T$ is finite and $\sup_{0\le t< T}|\gamma_i(t)|$ and $\sup_{0\le t< T}|\dot\gamma_i(t)|$
are finite for each $i=1,\dots, N$, this solution can be continued to $[0,T+\epsilon)$ for some $\epsilon>0$ (Chapter 1 of \cite{MR587488}). This would contradict that $[0,T)$ is the maximal interval of existence from which we would conclude that $T=\infty$. 

\par Furthermore, as
$$
\sup_{0\le t< T}|\gamma_i(t)|\le |x_i| + T\sup_{0\le t< T}|\dot\gamma_i(t)|,
$$
we focus on bounding $\sup_{0\le t< T}|\dot\gamma_i(t)|$.  To this end, we will employ the semiconvexity of $V$ and select $L\ge 0$ so that 
$$
(\R^d)^N\ni y\mapsto V(y)+\frac{L}{2}\sum^N_{i=1}m_i|y_i|^2
$$
is convex. We set $\gamma=(\gamma_1,\dots,\gamma_N)$, $x=(x_1,\dots, x_N)$ and note
\begin{align*}
V(\gamma(t))&\ge V(x)+DV(x)\cdot(\gamma(t)-x)-\frac{L}{2}\sum^N_{i=1}m_i|\gamma_i(t)-x_i|^2\\
 &= V(x)+\sum^N_{i=1}D_{x_i}V(x)\cdot(\gamma_i(t)-x_i)-\frac{L}{2}\sum^N_{i=1}m_i|\gamma_i(t)-x_i|^2\\
  &\ge V(x)+\sum^N_{i=1}\frac{D_{x_i}V(x)}{\sqrt{m_i}}\cdot \sqrt{m_i}(\gamma_i(t)-x_i)-\frac{L}{2}\sum^N_{i=1}m_i|\gamma_i(t)-x_i|^2\\
&\ge V(x)-\frac{1}{2}\sum^N_{i=1}\frac{\left|D_{x_i}V(x)\right|^2}{m_i}-\frac{1}{2}\left(L+1\right)\sum^N_{i=1}m_i|\gamma_i(t)-x_i|^2\\
&= V(x)-\frac{1}{2}\sum^N_{i=1}\frac{\left|D_{x_i}V(x)\right|^2}{m_i}-\frac{1}{2}\left(L+1\right)\sum^N_{i=1}m_i\left|\int^t_0\dot\gamma_i(s)ds\right|^2\\
&\ge V(x)-\frac{1}{2}\sum^N_{i=1}\frac{\left|D_{x_i}V(x)\right|^2}{m_i}-\frac{1}{2}\left(L+1\right)t\int^t_0\sum^N_{i=1}m_i\left|\dot\gamma_i(s)\right|^2ds.
\end{align*}

\par Substituting this lower bound for $V(\gamma(t))$ in the conservation of energy \eqref{ElemConsEnergy} gives  
\be\label{ConsIneq}
\sum^N_{i=1}m_i|\dot\gamma_i(t)|^2\le 
\sum^N_{i=1}m_i|v_i|^2 +\sum^N_{i=1}\frac{\left|D_{x_i}V(x)\right|^2}{m_i}+\left(L+1\right)t\int^t_0\sum^N_{i=1}m_i\left|\dot\gamma_i(s)\right|^2ds.
\ee
Observe
\begin{align*}
&\frac{d}{dt}\left\{e^{-\frac{1}{2}\left(L+1\right)t^2}\int^t_0\sum^N_{i=1}m_i|\dot\gamma_i(s)|^2ds\right\}\\
&\hspace{1in}=e^{-\frac{1}{2}\left(L+1\right)t^2}\left(\sum^N_{i=1}m_i|\dot\gamma_i(t)|^2-\left(L+1\right)t\int^t_0\sum^N_{i=1}m_i\left|\dot\gamma_i(s)\right|^2ds\right)\\
&\hspace{1in} \le e^{-\frac{1}{2}\left(L+1\right)t^2}\left(\sum^N_{i=1}m_i|v_i|^2 +\sum^N_{i=1}\frac{\left|D_{x_i}V(x)\right|^2}{m_i}\right),
\end{align*}
so that 
\be\label{PrimeEnergyEst2}
\int^t_0\sum^N_{i=1}m_i|\dot\gamma_i(s)|^2ds\le e^{\frac{1}{2}\left(L+1\right)t^2} \int^t_0e^{-\frac{1}{2}\left(L+1\right)s^2}ds\left(\sum^N_{i=1}m_i|v_i|^2 +\sum^N_{i=1}\frac{\left|D_{x_i}V(x)\right|^2}{m_i}\right).
\ee
Combining this inequality with \eqref{ConsIneq} gives
\begin{align}\label{PrimeEnergyEst}
&\sum^N_{i=1}m_i|\dot\gamma_i(t)|^2\le \\
& \left(1+
\left(L+1\right)t e^{\frac{1}{2}\left(L+1\right)t^2} \int^t_0e^{-\frac{1}{2}\left(L+1\right)s^2}ds\right)\left(\sum^N_{i=1}m_i|v_i|^2 +\sum^N_{i=1}\frac{\left|D_{x_i}V(x)\right|^2}{m_i}\right).
\end{align}
We conclude that $\sup_{0\le t< T}|\dot\gamma_i(t)|<\infty$ for any $T>0$, so the maximal interval of existence for this solution is $[0,\infty)$. 
\end{proof}

\subsection{Narrow convergence}
We now recall some important facts about the convergence of probability measures. Our primary reference for this material is the monograph by Ambrosio, Gigli, and Savar\'e \cite{AGS}. Let $(X, d)$ be a complete, separable metric space and ${\cal P}(X)$ denote the collection of Borel probability measures on $X$. 
Recall this space has a natural topology: $(\mu_k)_{k\in \N}\subset {\cal P}(X)$ {\it converges narrowly} to $\mu \in {\cal P}(X)$ 
provided 
\be\label{narrowConv}
\lim_{k\rightarrow\infty}\int_X gd\mu_k=\int_X gd\mu
\ee
for each $g$ belonging to $C_b(X)$, the space of bounded continuous functions on $X$.   We note that ${\cal P}(X)$ is metrizable. In particular, we can choose a metric of the form
\be\label{NarrowMetric}
\mathcal{d}(\mu,\nu):=\sum^\infty_{j=1}\frac{1}{2^j}\left|\int_{X}h_jd\mu-\int_{X}h_jd\nu\right|,\quad \mu,\nu\in {\cal P}(X).
\ee
Here each $h_j:X\rightarrow \R$ satisfies $\|h_j\|_\infty\le 1$ and Lip$(h_j)\le 1$ (Remark 5.1.1 of \cite{AGS}).
\par It will be important for us to be able to identify when a sequence of measures $(\mu_k)_{k\in \N}\subset {\cal P}(X)$ has a subsequence that converges narrowly. Fortunately, there is a convenient necessary and sufficient criterion which can be stated as follows. The sequence $(\mu_k)_{k\in \N}$ is precompact in the narrow topology if and only if there is a function $\varphi: X\rightarrow [0,\infty]$ with 
compact sublevel sets for which 
\be\label{Prokhorov}
\sup_{k\in \N}\int_{X}\varphi d\mu_k<\infty
\ee
(Remark 5.1.5 of \cite{AGS}).
\par We will also need to pass to the limit in \eqref{narrowConv} with functions $g$ which may not bounded. This naturally leads us to the notion of uniform integrability.   A Borel function 
$g: X\rightarrow [0,\infty]$ is said to be {\it uniformly integrable} with respect to the sequence $(\mu_k)_{k\in \N}$ 
provided 
$$
\lim_{R\rightarrow\infty}\int_{\left\{g\ge R\right\}}g d\mu_k=0
$$
uniformly in $k\in \N$.  It can be shown that if $(\mu_k)_{k\in \N}$ converges narrowly to $\mu$, $g: X\rightarrow \R$ is continuous and $|g|$ is uniformly integrable, then \eqref{narrowConv} holds (Lemma 5.1.7 in \cite{AGS}).

\section{Jeans-Vlasov equation}\label{JVsect}
The Jeans-Vlasov equation is
\be\label{JVeqn}
\partial_t f+v\cdot D_xf-D W*\rho \cdot D_v f=0,
\ee
which holds for all $(x,v,t)\in \R^d\times \R^d\times(0,\infty)$.  This equation provides a mean 
field description of a distribution of particles that interact pairwise via a force given by a potential energy $W$. Here $f$ represents the time dependent distribution of mass among all positions and velocity $(x,v)\in\R^d\times \R^d$ and $\rho$ represents the time dependent distribution of mass among all positions $x\in \R^d$.  Our main assumption will be that $W$ is semiconvex. Under this hypothesis, we will show that there is always one properly interpreted weak solution of \eqref{JVeqn} for a given initial mass distribution $f_0$ 
\be\label{fatzero}
f|_{t=0}=f_0.
\ee

\subsection{Weak solutions}
We note that any smooth solution $f=f(x,v,t)\ge 0$ of \eqref{JVeqn} with compact support in the $(x,v)$ variables preserves total mass in the sense that
$$
\frac{d}{dt}\int_{\R^d\times \R^d}f(x,v,t)dxdv=0.
$$
Consequently, we will suppose the total mass is initially equal to 1 and study solutions which take values in the space ${\cal P}(\R^d\times \R^d)$.  This leads to the following definition of a weak solution which specifies a type of measure valued solution.
\begin{defn}
Suppose $f_0\in{\cal P}(\R^d\times \R^d)$.  A {\it weak solution} of the Jeans-Vlasov equation \eqref{JVeqn} which satisfies the initial condition \eqref{fatzero} is a  narrowly continuous mapping $f: [0,\infty)\rightarrow {\cal P}(\R^d\times \R^d); t\mapsto f_t$ which satisfies
\be\label{WeakSolnCond}
\int^\infty_0\int_{\R^d\times \R^d}\left(\partial_t\phi+v\cdot D_x\phi -D W*\rho_t \cdot D_v\phi\right)df_tdt+\int_{\R^d\times \R^d}\phi|_{t=0}df_0=0
\ee
for each $\phi\in C^1_c(\R^d\times \R^d\times[0,\infty))$.  Here $\rho: [0,\infty)\rightarrow {\cal P}(\R^d); t\mapsto \rho_t$ is defined via
$$
\rho_t(B)=f_t(B\times \R^d)
$$
for each Borel subset $B\subset \R^d$. 
\end{defn}
\begin{rem}
Condition \eqref{WeakSolnCond} is equivalent to requiring that
\be\label{DistriCondJV}
\int_{\R^d\times \R^d}\psi df_t =\int_{\R^d\times \R^d}\psi df_0 
+\int^t_0\int_{\R^d\times \R^d}\left(v\cdot D_x\psi -D W*\rho_s \cdot D_v\psi\right)df_sds
\ee
for each $\psi\in C^1_c(\R^d\times \R^d)$ and $t\ge 0$.
\end{rem}

\par The Jeans-Vlasov system can be derived by first considered $N$ point masses in $\R^d$ that interactive pairwise by force given by $-D W$.  If $\gamma_1,\dots,\gamma_N$ describe the trajectories of these particles and $m_1,\dots, m_N>0$ are the respective masses with $\sum^N_{i=1}m_i=1$, the corresponding equations of motion are
\be\label{NewtonSystem}
m_i\ddot \gamma_i(t)=-\sum^N_{j= 1}m_im_jD W(\gamma_i(t)-\gamma_j(t)), \quad t>0
\ee
for $i=1,\dots, N$. We will assume throughout this section that
\be\label{Wassump}
\begin{cases}
W\in C^1(\R^d)\\\\
W(-z)=W(z)\; \text{for all}\; z\in \R^d\\\\
\R^d\ni z\mapsto W(z)+\dfrac{L}{2}|z|^2\;\;\text{is convex for some $L\ge 0$}
\\\\
\displaystyle\sup_{z\in \R^d}\displaystyle\frac{|DW(z)|}{1+|z|}<\infty.
\end{cases}
\ee
Note that as $W$ is even and $C^1$, $|DW(0)|=0$. Thus, the contribution to the force $DW(\gamma_i(t)-\gamma_j(t))$ in \eqref{NewtonSystem} vanishes for $j=i$.

\par It is also possible to write the \eqref{NewtonSystem} as
\be
m_i\ddot\gamma_i(t)=-D_{x_i}V(\gamma_1(t),\dots,\gamma_N(t)), \quad t>0
\ee
$(i=1,\dots, N)$, where 
$$
V(x)=\frac{1}{2}\sum^N_{i,j=1}m_im_jW(x_i-x_j), \quad x\in (\R^d)^N.
$$
It is evident that $V$ is continuously differentiable and the semiconvexity of $V$ follows from the identity
\be\label{VVlasovSemiCon}
V(x)+\frac{L}{2}\sum^N_{i=1}m_i|x_i|^2=\frac{1}{2}\sum^N_{i,j=1}m_im_j
\left(W(x_i-x_j) + \frac{L}{2}|x_i-x_j|^2\right)+\frac{L}{2}\left|\sum^N_{i=1}m_ix_i\right|^2.
\ee
In particular, the right hand side of \eqref{VVlasovSemiCon} is convex due to our assumption that $W(z)+(L/2)|z|^2$ is convex. By Proposition \ref{ODENewtonProp}, \eqref{NewtonSystem} has a solution $\gamma_1,\dots, \gamma_N\in C^2([0,\infty);\R^d)$ for any given set of initial conditions.

\par It also turns out that the paths $\gamma_1,\dots, \gamma_N$ generate a weak solution of the Jeans-Vlasov equation.
\begin{prop}\label{PropDiscreteVlasov}
Suppose $\gamma_1,\dots, \gamma_N\in C^2([0,\infty);\R^d)$ is a solution of \eqref{NewtonSystem} which satisfies
 \be\label{initialGamma}
\gamma_i(0)=x_i\quad \text{and}\quad \dot\gamma_i(0)=v_i.
\ee
Define $f: [0,\infty)\rightarrow  {\cal P}(\R^d\times \R^d); t\mapsto f_t$ as
\be\label{discreteWeakSolnJV}
f_t=\sum^N_{i=1}m_i\delta_{(\gamma_i(t),\dot\gamma_i(t))}
\ee
for each $t\ge 0$. Then $f$ is a weak solution of the Jeans-Vlasov equation \eqref{JVeqn} with initial condition 
\be\label{DiscreteFzero}
f_0=\sum^N_{i=1}m_i\delta_{(x_i,v_i)}.
\ee
\end{prop}
\begin{proof}
For $f_t$ defined by \eqref{discreteWeakSolnJV}, we have that its spatial marginal distribution is 
\be
\rho_t=\sum^N_{i=1}m_i\delta_{\gamma_i(t)}\in {\cal P}(\R^d)
\ee
for $t\ge 0$. It follows that
$$
\ddot\gamma_i(t)=-\sum^N_{j= 1}m_jD W(\gamma_i(t)-\gamma_j(t))=-\left(DW*\rho_t\right)(\gamma_i(t))
$$
for $t\ge 0$ and $i=1,\dots, N$.

\par As $\gamma_1,\dots, \gamma_N\in C^2([0,\infty);\R^d)$, $f: [0,\infty)\rightarrow {\cal P}(\R^d\times \R^d)$ is narrowly continuous. Moreover, for $\phi\in C^1_c(\R^d\times \R^d\times[0,\infty))$, 
\begin{align*}
&\int^\infty_0\int_{\R^d\times\R^d}\left(\partial_t\phi + v\cdot D_x\phi-D W*\rho_t\cdot D_v\phi\right)df_tdt\\
&=\int^\infty_0\sum^N_{i=1}m_i\left(\partial_t\phi(\gamma_i(t),\dot \gamma_i(t),t)+\dot \gamma_i(t)\cdot D_x\phi(\gamma_i(t),\dot \gamma_i(t),t)
-(DW*\rho_t)(\gamma_i(t))\cdot D_v\phi(\gamma_i(t),\dot \gamma_i(t),t)\right)dt\\
&=\int^\infty_0\sum^N_{i=1}m_i\left(\partial_t\phi(\gamma_i(t),\dot \gamma_i(t),t)+\dot \gamma_i(t)\cdot D_x\phi(\gamma_i(t),\dot \gamma_i(t),t)
+\ddot \gamma_i(t)\cdot D_v\phi(\gamma_i(t),\dot \gamma_i(t),t)\right)dt\\
&=\int^\infty_0\sum^N_{i=1}m_i\frac{d}{dt}\phi(\gamma_i(t),\dot \gamma_i(t),t)dt\\
&=-\sum^N_{i=1}m_i\phi(\gamma_i(0),\dot \gamma_i(0),0)\\
&=-\int_{\R^d\times\R^d}\phi|_{t=0}df_0.
\end{align*}
\end{proof}
\begin{rem}\label{NoNeedCompact}
Since the support of $f_t$ is a compact subset of $\R^d\times \R^d$, $f$ satisfies \eqref{WeakSolnCond} for each $\phi\in C^1(\R^d\times \R^d\times[0,\infty))$ and \eqref{DistriCondJV} for each $\psi\in C^1(\R^d\times \R^d)$. That is, 
we need not require the test functions to have compact support in $\R^d\times \R^d$. This detail will 
be useful when we consider the compactness of solutions of this type. 
\end{rem}
Each of the solutions we constructed above inherit a few moment estimates from the energy estimates satisfied by $\gamma_1,\dots,\gamma_N$.  In order to conveniently express these inequalities, we make use of the function 
\be\label{ChiFun}
\chi(t):=e^{\frac{1}{2}(L+1)t^2}\int^t_0e^{-\frac{1}{2}(L+1)s^2}ds, \quad t\ge 0.
\ee
We also note that $\chi$ is increasing with
$$
\dot\chi(t)=1+(L+1)te^{\frac{1}{2}(L+1)t^2}\int^t_0e^{-\frac{1}{2}(L+1)s^2}ds, \quad t\ge 0.
$$
\begin{prop}\label{KineticEnergyEstJV}
Define $f$ as in \eqref{discreteWeakSolnJV}. Then
\begin{align*}
\int_{\R^d\times\R^d}|v|^2df_t\le \left[\int_{\R^d\times \R^d}|v|^2df_0+\int_{\R^d\times \R^d}|DW(x-y)|^2d\rho_0(x)d\rho_0(y)\right]\dot\chi(t)
\end{align*}
for each $t\ge 0$. 
\end{prop}
\begin{proof}
As $\sum^N_{j=1}m_j=1$, 
$$
|D_{x_i}V(x)|^2=\left|\sum^N_{j= 1}m_jm_iD W(x_i-x_j)\right|^2\le \sum^N_{j= 1}m_jm_i^2|D W(x_i-x_j)|^2.
$$
Therefore, 
$$
\sum^N_{i=1}\frac{\left|D_{x_i}V(x)\right|^2}{m_i}\le \sum^N_{i,j=1}m_im_j|DW(x_i-x_j)|^2.
$$
And by \eqref{PrimeEnergyEst}, 
\begin{align*}
\int_{\R^d\times\R^d}|v|^2df_t& =\sum^N_{i=1}m_i|\dot\gamma_i(t)|^2\\
&\le \left[\sum^N_{i=1}m_i|v_i|^2 +\sum^N_{i=1}\frac{\left|D_{x_i}V(x)\right|^2}{m_i}\right]\dot\chi(t)\\
&\le \left[\sum^N_{i=1}m_i|v_i|^2+\sum^N_{i,j=1}m_im_j|DW(x_i-x_j)|^2\right]\dot\chi(t)\\
&=\left[\int_{\R^d\times \R^d}|v|^2df_0+\int_{\R^d\times \R^d}|DW(x-y)|^2d\rho_0(x)d\rho_0(y)
\right]\dot\chi(t)
\end{align*}
for $t\ge0$. 
\end{proof}
\begin{prop}\label{XsquaredEstJV}
Define $f$ as in \eqref{discreteWeakSolnJV}. Then
\begin{align*}
&\int_{\R^d\times\R^d}\frac{1}{2}|x|^2df_t\le \int_{\R^d\times\R^d}|x|^2 df_0+\left[\int_{\R^d\times\R^d}
|v|^2 df_0+\int_{\R^d\times\R^d}|DW(x-y)|^2d\rho_0(x)d\rho_0(y) \right]t\chi(t)
\end{align*}
for each $t\ge 0$. 
\end{prop}
\begin{proof}
Observe
\begin{align*}
\int_{\R^d\times\R^d}\frac{1}{2}|x|^2 df_t& = \sum^N_{i=1}\frac{m_i}{2}|\gamma_i(t)|^2\\
  & = \sum^N_{i=1}\frac{m_i}{2}\left|x_i+\int^t_0\dot\gamma_i(s)ds\right|^2\\
  & \le  \sum^N_{i=1}m_i\left\{|x_i|^2+t\int^t_0|\dot\gamma(s)|^2ds\right\}\\
    & = \sum^N_{i=1}m_i|x_i|^2+t\int^t_0\sum^N_{i=1}m_i|\dot\gamma(s)|^2ds.
    \end{align*}
And in view of \eqref{PrimeEnergyEst2},
\begin{align}\label{IntegralBounddotGammaW}
\int^t_0\sum^N_{i=1}m_i|\dot\gamma_i(s)|^2ds\le \left[\sum^N_{i=1}m_i|v_i|^2+\sum^N_{i,j=1}m_im_j|DW(x_i-x_j)|^2\right]\chi(t).
\end{align}
Consequently, 
\begin{align}
\int_{\R^d\times\R^d}\frac{1}{2}|x|^2 df_t&\le  \sum^N_{i=1}m_i|x_i|^2 + \left[\sum^N_{i=1}m_i|v_i|^2+\sum^N_{i,j=1}m_im_j|DW(x_i-x_j)|^2\right]t\chi(t)\\
&=\int_{\R^d\times\R^d}|x|^2 df_0+\left[\int_{\R^d\times\R^d}
|v|^2 df_0+\int_{\R^d\times\R^d}|DW(x-y)|^2d\rho_0(x)d\rho_0(y) \right]t\chi(t).
\end{align}
\end{proof}

\subsection{Compactness}
Recall that every $f_0\in {\cal P}(\R^d\times \R^d)$ is a limit of a sequence of measures of the form \eqref{DiscreteFzero}. That is, convex combinations of Dirac measures are dense in ${\cal P}(\R^d\times \R^d)$. Therefore, we can choose a sequence $(f_0^k)_{k\in \N}\subset {\cal P}(\R^d\times \R^d)$ for which 
\be\label{FzeroSeq}
\begin{cases}
f^k_0\rightarrow f_0 \;\text{in}\; {\cal P}(\R^d\times \R^d)\\\\
f^k_0\;\text{is of the form \eqref{DiscreteFzero}}.
\end{cases}
\ee
We will additionally suppose that 
\be\label{InitJVE}
\int_{\R^d\times \R^d}(|x|^2+|v|^2)df_0<\infty
\ee
and choose the sequence of approximate initial conditions $(f_0^k)_{k\in \N}$ to satisfy \eqref{FzeroSeq} and 
\be\label{strongconvTimezero}
\lim_{k\rightarrow\infty}\int_{\R^d\times \R^d}(|x|^2+|v|^2)df^k_0=
\int_{\R^d\times \R^d}(|x|^2+|v|^2)df_0.
\ee
It is well known that this can be accomplished; we refer the reader to \cite{Bolley} for a short proof of this fact. 

\par By Proposition \ref{PropDiscreteVlasov}, we then have a sequence of weak solutions $(f^k)_{k\in \N}$ of the Jeans-Vlasov equation \eqref{JVeqn} with initial conditions $f^k|_{t=0}=f^k_0$ for $k\in \N$. Our goal is then to show that this sequence has a subsequence that converges to a solution $f$ of the Jean-Vlasov equation with initial condition $f_0$.  The following compactness lemma will gives us a candidate for a weak solution. 
\begin{lem}\label{CompactnessLemmaJV}
There is a subsequence $(f^{k_j})_{j\in \N}$ and $f:[0,\infty)\rightarrow {\cal P}(\R^d\times \R^d)$ such that 
$$
f_t^{k_j}\rightarrow f_t\;\textup{in}\;{\cal P}(\R^d\times \R^d)
$$
uniformly for $t$ belonging to compact subintervals of $[0,\infty)$.  Moreover, 
\be\label{W1convergence}
\lim_{j\rightarrow\infty}\int_{\R^d\times \R^d}\varphi df^{k_j}_t=
\int_{\R^d\times \R^d}\varphi df_t
\ee
for each $t\ge 0$ and continuous $\varphi:\R^d\times \R^d\rightarrow \R$ which satisfies 
\be\label{1growthPhi}
\sup_{x,v\in \R^d}\frac{|\varphi(x,v)|}{1+|x|+|v|}<\infty.
\ee
\end{lem}

\begin{proof}
In view of the last line in \eqref{Wassump} and the limit \eqref{strongconvTimezero}, 
\be\label{DWsqIntFinite}
\sup_{k\in \N}\int_{\R^d}\int_{\R^d}|DW(x-y)|^2d\rho^k_0(x)d\rho^k_0(y)<\infty.
\ee
We also have  
\be\label{BasicMomentBoundxandv}
\sup_{0\le t\le T}\sup_{k\in \N}\int_{\R^d\times \R^d}(|x|^2+|v|^2)df^k_t<\infty
\ee
for each $T\ge 0$. This finiteness follows from \eqref{strongconvTimezero}, \eqref{DWsqIntFinite} and Propositions \ref{KineticEnergyEstJV} and \ref{XsquaredEstJV}. 

\par Moreover, there is a constant $C$ such that
\begin{align*}
 \int_{\R^d\times \R^d}|D W*\rho^k_t|df^k_t&= \int_{\R^d}\int_{\R^d}|DW(x-y)|d\rho^k_t(x)d\rho^k_t(y)\\
&\le \int_{\R^d}\int_{\R^d}C(1+|x|+|y|)d\rho^k_t(x)d\rho^k_t(y)\\
&\le C\left[1+ 2\int_{\R^d}|x|d\rho^k_t\right]\\
&= C\left[1+ 2\int_{\R^d\times \R^d}|x|df^k_t\right]\\
&\le C\left[1+ 2\left(\int_{\R^d\times \R^d}|x|^2df^k_t\right)^{1/2}\right].
\end{align*}
It then follows that 
$$
\sup_{0\le t\le T}\sup_{k\in \N} \int_{\R^d\times \R^d}|D W*\rho^k_t|df^k_t<\infty
$$
for each $T>0$. 

\par Note
\be
\int_{\R^d\times \R^d}\psi df^k_t-\int_{\R^d\times \R^d}\psi df^k_s=\int^t_s\int_{\R^d\times \R^d}\left(v\cdot D_x\psi -D W*\rho^k_\tau \cdot D_v\psi\right)df^k_\tau d\tau
\ee
for $0\le s\le t$ and $\psi\in C^1(\R^d\times \R^d)$; here we recall Remark \ref{NoNeedCompact}. Thus, for $0\le s\le t\le T$ there is a constant $C(T)>0$ such that
\begin{align}\label{LipPsiLipEst}
\left|\int_{\R^d\times \R^d}\psi df^k_t-\int_{\R^d\times \R^d}\psi df^k_s\right|&\le\text{Lip}(\psi)
\int^t_s\left(\int_{\R^d\times \R^d}(|v|+|D W*\rho^k_\tau|)df^k_\tau\right) d\tau\\
&\le\text{Lip}(\psi)C(T)(t-s)
\end{align}
for each $k\in\N$. This estimate actually holds for all Lipschitz continuous $\psi:\R^d\times \R^d\rightarrow \R$. Indeed, we can smooth such a $\psi$ with a mollifier, verify \eqref{LipPsiLipEst} with the mollification of $\psi$ and pass to the limit to discover that the inequality holds for $\psi$. We leave the details to the reader as they involve fairly standard computations.   

\par We conclude,  that for each $T>0$ 
$$
\sup_{k\in \N}\mathcal{d}(f^k_t,f^k_s)\le C(T)(t-s), \quad 0\le s\le t\le T.
$$
Here $\mathcal{d}$ is the metric defined in \eqref{NarrowMetric}. In particular, we recall that $\mathcal{d}$ metrizes the narrow topology on ${\cal P}(\R^d\times \R^d)$. We can also appeal to \eqref{BasicMomentBoundxandv} to conclude that 
 that $(f^k_t)_{k\in \N}$ is precompact in ${\cal P}(\R^d\times \R^d)$ for each $t\ge 0$. The Arzel\`a-Ascoli theorem then implies that there is a narrowly continuous $f:[0,\infty)\rightarrow {\cal P}(\R^d\times \R^d)$ and a subsequence $(f_t^{k_j})_{j\in \N}$ such that $f_t^{k_j}\rightarrow f_t$ in ${\cal P}(\R^d\times \R^d)$ uniformly for $t$ belonging to compact subintervals of $[0,\infty)$ as $j\rightarrow\infty$. 

 \par We now will verify \eqref{W1convergence}. Fix $t\ge 0$ and suppose $\varphi$ satisfies \eqref{1growthPhi}. Choose a constant $A$ so that $|\varphi(x,v)|\le A(|x|+|v|+1)$ and observe
\begin{align*}
\int_{|\varphi|\ge R}|\varphi| df^k_t&\le A \int_{A(|x|+|v|+1)\ge R}(|x|+|v|+1)df^k_t\\
&\le A\int_{A\sqrt{3}\sqrt{|x|^2+|v|^2+1}\ge R}(\sqrt{3}\sqrt{|x|^2+|v|^2+1})df^k_t\\
&\le\frac{3A^2}{R} \int_{\sqrt{3}\sqrt{|x|^2+|v|^2+1}\ge R}(|x|^2+|v|^2+1)df^k_t\\
&\le\frac{3A^2}{R} \int_{\R^d\times \R^d}(|x|^2+|v|^2+1)df^k_t.
\end{align*}
Therefore, 
$$
\lim_{R\rightarrow\infty}\int_{|\varphi|\ge R}|\varphi| df^k_t=0
$$
uniformly in $k\in \N$. As a result, $|\varphi|$ is uniformly integrable with respect to the narrowly convergence sequence $(f_t^{k_j})_{j\in \N}$ and \eqref{W1convergence} follows. 
\end{proof}
We are now in a position to verify the existence of solutions to the Jeans-Vlasov equation with a semiconvex potential. This result was first established by Ambrosio and Gangbo via a time discretization scheme \cite{MR2361303}; Kim also verified this result by using an inf-convolution regularization method \cite{MR3017033}.  Our approach is distinct in that it uses particle trajectories, although it applies to a smaller class of problems than the ones considered in \cite{MR2361303} and \cite{MR3017033}. We also mention the survey by Jabin \cite{MR3317577} which discusses mean field limits related to the Jeans-Vlasov equation.
\begin{thm}
Assume $f_0$ satisfies \eqref{InitJVE}. Any mapping $f: [0,\infty)\rightarrow{\cal P}(\R^d\times \R^d)$ as obtained in Lemma \ref{CompactnessLemmaJV} is a weak solution of the Jeans-Vlasov equation \eqref{JVeqn} which satisfies the initial condition \eqref{fatzero}.
\end{thm}

\begin{proof}
Fix $\phi\in C^1_c(\R^d\times \R^d\times[0,\infty))$. We have
\be\label{weakSolnCondJay}
\int^\infty_0\int_{\R^d\times \R^d}\left(\partial_t\phi+v\cdot D_x\phi -D W*\rho^{k_j}_t \cdot D_v\phi\right)df^{k_j}_tdt+\int_{\R^d\times \R^d}\phi|_{t=0}df^{k_j}_0=0
\ee
for each $j\in \N$. We will argue that we can send $j\rightarrow \infty$ in this equation and replace $f^{k_j}$ with $f$, which will show that $f$ is the desired weak solution. 

\par First note
$$
\lim_{j\rightarrow\infty}\int_{\R^d\times \R^d}\phi|_{t=0}df^{k_j}_0=\int_{\R^d\times \R^d}\phi|_{t=0}df_0,
$$
by \eqref{FzeroSeq}. Since $f^{k_j}_t\rightarrow f_t$ narrowly for $t$ belonging to compact subintervals of $[0,\infty)$, we also have 
$$
\lim_{j\rightarrow\infty}\int^\infty_0\int_{\R^d\times \R^d}\left(\partial_t\phi+v\cdot D_x\phi\right) df^{k_j}_tdt
=\int^\infty_0\int_{\R^d\times \R^d}\left(\partial_t\phi+v\cdot D_x\phi\right)\phi df_tdt
$$
as the function $\partial_t\phi+v\cdot D_x\phi$ is bounded, continuous and compactly supported. 

\par  Next observe  
\begin{align*}
&\int^\infty_0\int_{\R^d\times \R^d}D W*\rho^{k_j}_t \cdot D_v\phi df^{k_j}_tdt\\
&\quad=\int^\infty_0\int_{\R^d\times \R^d}\left(\int_{\R^d}D W(x-y)\rho^{k_j}_t(y)\right) \cdot D_v\phi(x,v) df^{k_j}_t(x,v)dt\\
&\quad=\int^\infty_0\int_{\R^d\times \R^d}\int_{\R^d\times \R^d}\left[D W(x-y)\cdot D_v\phi(x,v) \right]df^{k_j}_t(x,v)df^{k_j}_t(y,w)dt.
\end{align*}
By assumption, $|D W(x-y)|\le C(1+|x|+|y|)$ for some $C\ge 0$. Therefore, 
\be\label{DWandDPhiIneq}
|D W(x-y)\cdot D_v\phi(x,v)|\le C\|D_v\phi\|_\infty(1+|x|+|y|)
\ee
for $x,y,v\in \R^d$. 

\par Note that $1+|x|+|y|$ is uniformly integrable with respect to  $(f^{k_j}_t\times f^{k_j}_t)_{j\in \N}$. Indeed, 
\begin{align*}
\iint_{1+|x|+|y|\ge R}(1+|x|+|y|)d(f^{k_j}_t\times f^{k_j}_t)&\le\frac{1}{R} \iint_{1+|x|+|y|\ge R}(1+|x|+|y|)^2d(f^{k_j}_t\times f^{k_j}_t)\\
 &\le\frac{3}{R} \iint_{1+|x|+|y|\ge R}(1+|x|^2+|y|^2)d(f^{k_j}_t\times f^{k_j}_t)\\
 &\le\frac{3}{R} \int_{(\R^d\times \R^d)^2}(1+|x|^2+|y|^2)d(f^{k_j}_t\times f^{k_j}_t) \\
 &\le\frac{3}{R} \left(1+2\int_{\R^d\times \R^d}|x|^2df^{k_j}_t \right)
\end{align*}
which goes to zero as $R\rightarrow\infty$ uniformly in $j$. The uniform integrability of $(x,y,v)\mapsto |D W(x-y)\cdot D_v\phi(x,v)|$ follows by \eqref{DWandDPhiIneq}.

\par As $f^{k_j}_t\rightarrow f_t$ in ${\cal P}(\R^d\times \R^d)$ for each $t\ge 0$, $f^{k_j}_t\times f^{k_j}_t \rightarrow f_t \times f_t$ in ${\cal P}((\R^d\times \R^d)^2)$ for each $t\ge 0$ (Theorem 2.8 in \cite{Billingsley}). Therefore,
$$
\lim_{j\rightarrow\infty}\int_{\R^d\times \R^d}D W*\rho^{k_j}_t \cdot D_v\phi df^{k_j}_t=
\int_{\R^d\times \R^d}D W*\rho_t \cdot D_v\phi df_t
$$
for each $t\ge 0$. Since the sequence of functions 
$$
[0,\infty)\ni t\mapsto \int_{\R^d\times \R^d}D W*\rho^{k_j}_t \cdot D_v\phi df^{k_j}_t
$$
are all supported in a common interval and can be bounded independently of $j\in \N$ by \eqref{BasicMomentBoundxandv}, we can apply dominated convergence to conclude
$$
\lim_{j\rightarrow\infty}\int^\infty_0\int_{\R^d\times \R^d}D W*\rho^{k_j}_t \cdot D_v\phi df^{k_j}_tdt
=\int^\infty_0\int_{\R^d\times \R^d}D W*\rho_t \cdot D_v\phi df_tdt.
$$
We finally send $j\rightarrow\infty$ in \eqref{weakSolnCondJay} to deduce that is $f$ is a weak solution of the Jeans-Vlasov equation with $f|_{t=0}=f_0$.  
\end{proof}

\subsection{Quadratic interaction potentials}
A typical family of semiconvex interaction potentials is 
\be
W(z)=\frac{\kappa}{2}|z|^2,
\ee
where $\kappa\in \R$. It turns out that it is easy to write down an explicit solution for each member of this family. To this end, 
we let $f_0\in {\cal P}(\R^d\times \R^d)$ and make use of the projection maps
$$
\pi^1:\R^d\times \R^d\rightarrow\R^d; (x,v)\mapsto x
$$
and 
$$
\pi^2:\R^d\times \R^d\rightarrow\R^d; (x,v)\mapsto v.
$$
\par If $\kappa>0$, we set 
$$
X(x,v,t)=\left(x-\int_{\R^d\times\R^d}\pi^1df_0\right)\cos(\sqrt{\kappa}t)+\left(v-\int_{\R^d\times\R^d}\pi^2df_0\right)\frac{1}{\sqrt{\kappa}}\sin(\sqrt{\kappa}t)+\int_{\R^d\times \R^d}(\pi^1+t\pi^2)df_0.
$$
for $(x,v,t)\in \R^d\times \R^d\times[0,\infty)$.  For each $t\ge 0$, we then define $f_t\in {\cal P}(\R^d\times \R^d)$ via the formula:  
$$
\int_{\R^d\times \R^d}\varphi df_t=\int_{\R^d\times\R^d}\varphi(X(x,v,t),\partial_tX(x,v,t))df_0(x,v)
$$
for $\varphi\in C_b(\R^d\times \R^d)$.  It is routine to check that 
$$
\partial_t^2X(x,v,t)=-[DW*\rho_t](X(x,v,t))
$$
and in particular that $f: [0,\infty)\rightarrow {\cal P}(\R^d\times \R^d); t\mapsto f_t$ is a weak solution of the Jeans-Vlasov equation
with $f|_{t=0}=f_0$. 

\par We can reason similarly when $\kappa<0$. Indeed, we could repeat the process above to build a weak solution via the map 
\begin{align*}
&X(x,v,t)=\\
&\left(x-\int_{\R^d\times\R^d}\pi^1df_0\right)\cosh(\sqrt{-\kappa}t)+\left(v-\int_{\R^d\times\R^d}\pi^2df_0\right)\frac{1}{\sqrt{\kappa}}\sinh(\sqrt{-\kappa}t)+\int_{\R^d\times \R^d}(\pi^1+t\pi^2)df_0.
\end{align*}
Finally, when $\kappa=0$ we can argue as above with the map
$$
X(x,v,t)=x+tv.
$$

\section{Pressureless Euler equations}\label{PEsect}
We now turn our attention to the pressureless Euler equations in one spatial dimension
\be\label{PressLessEuler}
\begin{cases}
\hspace{.27in}\partial_t\rho+\partial_x(\rho v)=0\\
\partial_t(\rho v)+\partial_x(\rho v^2)=-\rho(W'*\rho),
\end{cases}
\ee
which hold in $\R\times (0,\infty)$. These equations model the dynamics of a collection of 
particles restricted move on a line that interact pairwise via a potential $W:\R\rightarrow \R$ and 
undergo perfectly inelastic collisions once they collide. In particular, after particles collide, they 
remain stuck together. We note that the first equation in \eqref{PressLessEuler} states the 
conservation of mass and the second equation asserts conservation of momentum.
The unknowns are the distribution of particles $\rho$ and the corresponding velocity field $v$.  

\par Our goal is to verify the existence of a weak solution pair for given conditions
\be\label{Init}
\rho|_{t=0}=\rho_0\quad \text{and}\quad v|_{t=0}=v_0.
\ee
As with the Jeans-Vlasov equation, it will be natural for us to work with the space ${\cal P}(\R)$ of Borel probability 
measures on $\R$.   For convenience, we will always suppose that $\rho_0\in {\cal P}(\R)$ has finite second moment
\be\label{InitRhoZeroVeeZero}
\int_{\R}x^2d\rho_0(x)<\infty
\ee
and that
\be
v_0: \R\rightarrow \R\;\text{is absolutely continuous}.
\ee

\begin{defn}\label{WeakSolnDefn}
A narrowly continuous $\rho: [0,\infty)\rightarrow {\cal P}(\R); t\mapsto \rho_t$ and a Borel measurable
$v:\R\times[0,\infty)\rightarrow\R$ is a {\it weak solution pair of \eqref{PressLessEuler}
which satisfies the initial conditions \eqref{Init}} if for each $\phi\in C^\infty_c(\R\times[0,\infty))$,
$$
\int^\infty_0\int_{\R}(\partial_t\phi+v\partial_x\phi)d\rho_tdt+\int_{\R}\phi(\cdot,0)d\rho_0=0
$$
and
$$
\int^\infty_0\int_{\R}(v\partial_t\phi +v^2\partial_x\phi)d\rho_tdt+\int_{\R}\phi(\cdot,0)v_0d\rho_0=\int^\infty_0\int_{\R}\phi(W'*\rho_t)d\rho_tdt.
$$
\end{defn}

\par We will also assume that $W$ satisfies \eqref{Wassump} and in particular that 
$$
\R\ni x\mapsto W(x)+\frac{L}{2}x^2
$$
is convex for some $L>0$. We acknowledge that we already proved the following existence theorem in a recent preprint \cite{Hynd}. Nevertheless, we have included this material in this paper to illustrate how it fits into the bigger picture of evolution equations from physics with semiconvex potentials. We hope to provide just enough details so the reader has a good understanding of the main ideas.

\begin{thm}\label{EPthm} There is a weak solution pair $\rho$ and $v$ of \eqref{PressLessEuler} which satisfy the initial conditions \eqref{Init}. Moreover, 
$$
\int_{\R}\frac{1}{2}v(x,t)^2d\rho_t(x)+\iint_{\R^2}\frac{1}{2}W(x-y)d\rho_t(x)d\rho_t(y)\le \int_{\R}\frac{1}{2}v(x,s)^2d\rho_s(x)+\iint_{\R^2}\frac{1}{2}W(x-y)d\rho_s(x)d\rho_s(y)
$$
for Lebesgue almost every $0\le s\le t$ and 
\be\label{entropy}
(v(x,t)-v(y,t))(x-y)\le\frac{\sqrt{L}}{\tanh(\sqrt{L}t)}(x-y)^2
\ee
for $\rho_t$ almost every $x,y\in \R$ and Lebesgue almost every $t>0$. 
\end{thm}

\par We also mention that this theorem complements the fundamental existence results for equations which arise when studying sticky particle dynamics with interactions such as  \cite{BreGan, ERykovSinai}. While our discussion here doesn't include the prototypical nonsmooth potential $W(x)=|x|$, we note that such potentials can be considered by fairly standard approximation arguments (see section 5 of \cite{Hynd}).  Other notable works along these lines include \cite{BreGre,MR3296602,Dermoune,GNT,Guo,MR4007615,Jin,LeFloch,NatSav,NguTud,Shen}.

\subsection{Lagrangian coordinates}
Our approach to solving the pressureless Euler equations for given initial conditions is to find an absolutely continuous $X: [0,\infty)\rightarrow L^2(\rho_0)$ 
which satisfies
\be\label{FlowMapEqn}
\dot X(t)=\E_{\rho_0}\left[v_0-\displaystyle\int^t_0(W'*\rho_s)(X(s))ds\bigg| X(t)\right], \quad a.e.\;t\ge 0
\ee
and
\be\label{xInit}
X(0)=\id.
\ee
Here the Borel probability measure
\be\label{PushForwardMeasure}
\rho_t:=X(t)_{\#}\rho_0, \quad t\ge 0
\ee
is defined through the formula 
$$
\int_{\R}hd\rho_t=\int_{\R}h\circ X(t)d\rho_0
$$
for $h\in C_b(\R)$. 

\par The expression $\E_{\rho_0}[ g |X(t)]$ represents the usual conditional expectation of a Borel $g:\R\rightarrow \R$ given $X(t)$.  In particular, $X$ solves \eqref{FlowMapEqn} if there is a Borel $v: \R\times[0,\infty)\rightarrow\R$ for which 
\be\label{xODEv}
\dot X(t)=v(X(t),t)
\ee
and if
$$
\int_{\R}h(X(t))v(X(t),t)d\rho_0=\int_{\R}h(X(t))\left[v_0-\displaystyle\int^t_0(W'*\rho_s)(X(s))ds\right]d\rho_0
$$ 
for almost every $t\ge 0$ and each $h\in C_b(\R)$.  It is an elementary exercise to verify the following lemma. We leave the details to the reader. 

\begin{lem}\label{solnPElemma}
Suppose $X: [0,\infty)\rightarrow L^2(\rho_0)$ is absolutely continuous and satisfies \eqref{FlowMapEqn} and \eqref{xInit}. Define $\rho: [0,\infty)\rightarrow {\cal P}(\R); t\mapsto X(t){_\#}\rho_0$ and select a Borel $v: \R\times[0,\infty)\rightarrow\R$ for which \eqref{xODEv} holds for almost everywhere $t\ge 0$. Then $\rho$ and $v$ is a weak solution pair of \eqref{PressLessEuler} which satisfies the initial conditions \eqref{Init}.
\end{lem}

\par As a result, in order to design a solution of the pressureless Euler equations for given initial conditions, we only need to find an absolutely continuous $X: [0,\infty)\rightarrow L^2(\rho_0)$ which satisfies \eqref{FlowMapEqn} and \eqref{xInit}. We will do so by generating a solution when $\rho_0$ is a convex combination of Dirac 
measures and then showing how this can be extended to a general $\rho_0$ by a compactness argument.  The key to our compactness will of course be in exploiting the semiconvexity of $W$.

\subsection{Sticky particle trajectories}
 Let us initially suppose that 
 \be\label{ConvCombDirac}
 \rho_0=\sum^N_{i=1}m_i\delta_{x_i},
 \ee
 where $x_1,\dots, x_N$ are distinct and $m_1,\dots, m_N>0$ with $\sum^N_{i=1}m_i=1$.  Here $x_i$ represents the initial position of the particle with mass $m_i$. We will construct a solution of  \eqref{FlowMapEqn} with trajectories $\gamma_1,\dots, \gamma_N:[0,\infty)\rightarrow \R$ that evolve in time according to Newton's second law 
\be\label{NewtonSystem}
\ddot \gamma_i(t)=-\sum^N_{j=1}m_jW'(\gamma_i(t)-\gamma_j(t))
\ee
on any subinterval of $(0,\infty)$ where there is not a collision between particles.  When particles do collide, they experience perfectly inelastic collisions.  For example, if the subcollection of particles with masses $m_1,\dots, m_k$ collide at time $s>0$, they merge to form a single particle of mass $m_1+\dots + m_k$ and 
$$
m_1\dot\gamma_1(s-)+\dots +m_k\dot\gamma_k(s-)=(m_1+\dots +m_k)\dot\gamma_i(s+)
$$
for $i=1,\dots, k$.    In particular, the paths $\gamma_1,\dots, \gamma_k$ all agree after time $s$; see Figure \ref{fourMass}. These paths are known as {\it sticky particle trajectories} and they satisfy the following basic properties.

\begin{figure}[h]
\centering
 \includegraphics[width=.75\textwidth]{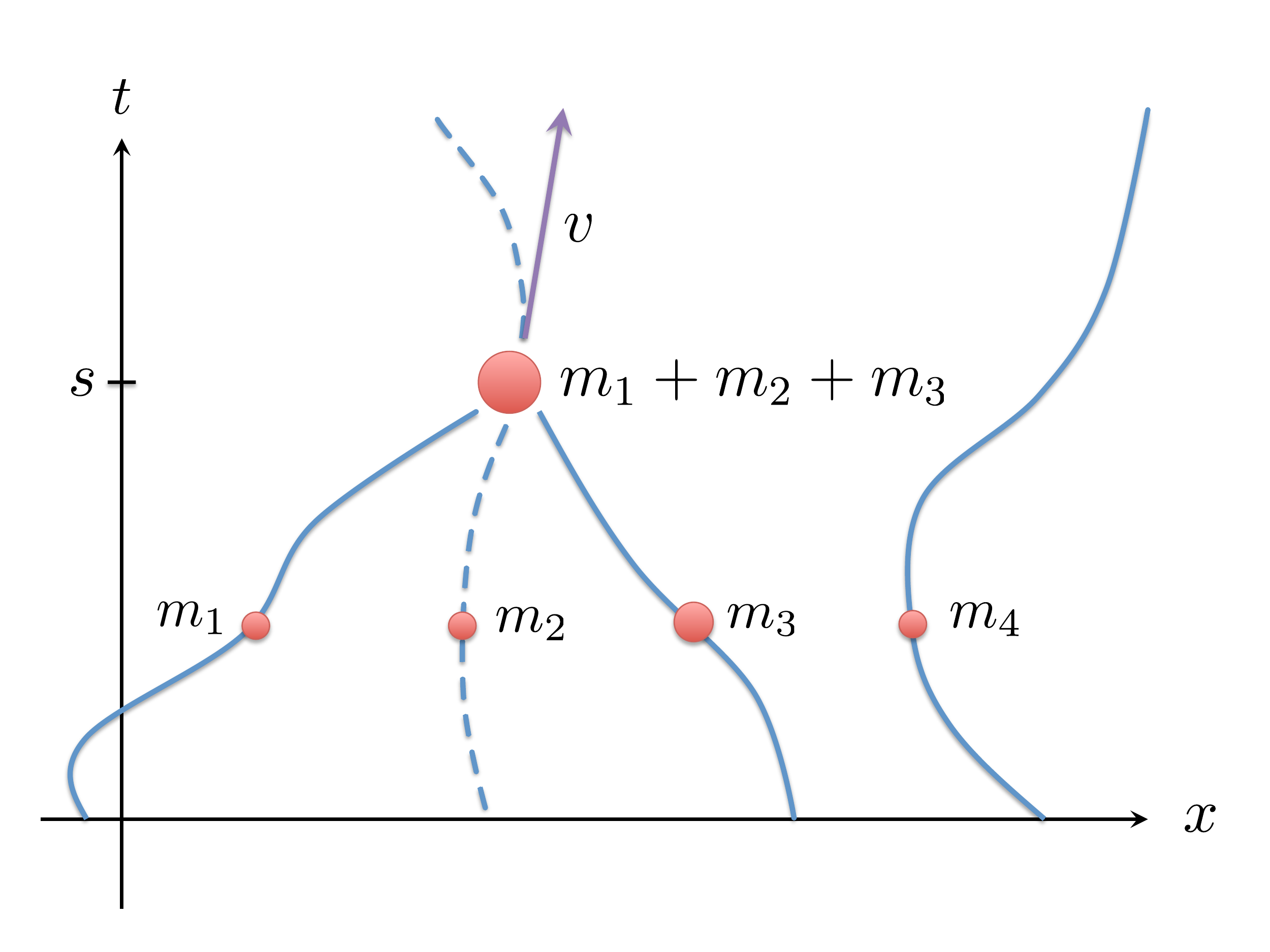}
 \caption{Schematic of the sticky particle trajectories $\gamma_1,\gamma_2, \gamma_3$ and $\gamma_4$ with respective point masses $m_1, m_2,m_3,$ and $m_4$. Trajectory $\gamma_2$ is shown in dashed and the masses are displayed larger than points to emphasize that they are allowed to be distinct. Note that the point masses $m_1, m_2$ and $m_3$ undergo a perfectly inelastic collision at time $s$.  Here the slope $v$ satisfies $ m_1\dot\gamma_1(s-)+m_2\dot\gamma_2(s-) +m_3\dot\gamma_3(s-)=(m_1+m_2 +m_3)v$, in accordance with the rules of perfectly inelastic collisions. }\label{fourMass}
\end{figure}

\begin{prop}\label{StickyParticlesExist}
 There are piecewise 
$C^2$ paths 
$$
\gamma_1,\dots,\gamma_N : [0,\infty)\rightarrow \R
$$
with the following properties. \\
(i) For $i=1,\dots, N$ and all but finitely many $t\in (0,\infty)$, \eqref{NewtonSystem} holds. 
\\
(ii) For $i=1,\dots, N$,
$$
\gamma_i(0)=x_i\quad \text{and}\quad \dot\gamma_i(0)=v_0(x_i).
$$
(iii) For $i,j=1,\dots, N$, $0\le s\le t$ and $\gamma_i(s)=\gamma_j(s)$ imply 
$$
\gamma_i(t)=\gamma_j(t).
$$
(iv) If $t>0$, $\{i_1,\dots, i_k\}\subset\{1,\dots, N\}$, and
$$
\gamma_{i_1}(t)=\dots=\gamma_{i_k}(t)\neq \gamma_i(t)
$$
for $i\not\in\{i_1,\dots, i_k\}$, then
$$
\dot\gamma_{i_j}(t+)=\frac{m_{i_1}\dot\gamma_{i_1}(t-)+\dots+m_{i_k}\dot\gamma_{i_k}(t-)}{m_{i_1}+\dots+m_{i_k}}
$$
for $j=1,\dots, k$.
\end{prop}
\begin{proof}
We argue by induction on $N$.  The assertion is trivial to verify for $N=1$ as there are no collisions 
and the lone trajectory is linear $\gamma_1(t)=x_1+t v_1$.  When $N>1$, we can solve the ODE system \eqref{NewtonSystem} for the given 
initial conditions in $(ii)$ and obtain trajectories $\xi_1, \dots,\xi_N\in C^2([0,\infty);\R)$; this follows from Proposition \ref{ODENewtonProp}. If these trajectories do not intersect, 
we set $\gamma_i=\xi_i$ for $i=1,\dots, N$ and conclude. If they do intersect for the first time at $s>0$, we can use the induction hypothesis. 

\par To this end, let us suppose initially that only one subcollection of these trajectories intersect first time at $s$. That is, there is a subset $\{i_1,\dots, i_k\}\subset\{1,\dots, N\}$ such that
$$
x:=\xi_{i_1}(s)=\dots=\xi_{i_k}(s)\neq \xi_i(s)
$$
for $i\not\in\{i_1,\dots, i_k\}$.  We also define
$$
v:=\frac{m_{i_1}\dot\xi_{i_1}(s-)+\dots+m_{i_k}\dot\xi_{i_k}(s-)}{m_{i_1}+\dots+m_{i_k}}.
$$

\par Observe that there are $N-k+1$ distinct positions $\{\xi_i(s)\}_{i\neq i_j}$ and $x$ at time $s$; there are also the  velocities  
$\{\dot\xi_i(s)\}_{i\neq i_j}$ and $v$ and masses $\{m_i\}_{i\neq i_j}$ and $m_{i_1}+\dots+m_{i_k}$ which correspond to these positions at time $s$. By our induction hypothesis, there are $N-k+1$ sticky particle trajectories $\{\zeta_i\}_{i\neq i_j}$ and $\zeta$ with these respective initial positions, initial velocities and masses. We can then set
$$
\gamma_i(t)=
\begin{cases}
\xi_i(t), \quad &0\le t\le s\\
\zeta_i(t-s), \quad & s\le t\le \infty
\end{cases}
$$
for $i\neq i_j$ and 
$$
\gamma_{i_j}(t)=
\begin{cases}
\xi_{i_j}(t), \quad &0\le t\le s\\
\zeta(t-s), \quad & s\le t\le \infty
\end{cases}
$$
for $j=1,\dots, k$.  It is routine to check that $\gamma_1\dots,\gamma_N$ are piecewise $C^2$ and satisfy $(i)-(iv)$. Moreover, it is straighforward to generalize this argument to the case where there are more than one subcollection of trajectories that intersect for the first at $s$. 
\end{proof}

\par Perhaps the most subtle property of sticky particle trajectories is the {\it averaging property}. This feature follows from Proposition \ref{StickyParticlesExist} parts $(iii)$ and $(iv)$  and is stated as follows.
\begin{cor}\label{AveProp}
Assume $g:\R\rightarrow \R$ and $0\le s<t$. Then 
\begin{align}\label{AveragingProp}
\sum^N_{i=1}m_ig(\gamma_i(t))\dot\gamma_i(t+)=  \sum^N_{i=1}m_ig(\gamma_i(t))\left[\dot\gamma_i(s+)-\int^t_s \sum^N_{j=1}m_jW'(\gamma_i(\tau)-\gamma_j(\tau))d\tau \right].
\end{align}
\end{cor}
\noindent This identity may seem curious at first sight. However, it turns out to be quite natural.  Indeed, it asserts that the ODE system \eqref{NewtonSystem} holds in a conditional sense. In particular, we will see that it encodes the conservation of momentum that occurs in between and during collisions.

\par There are also a few more bounds that will be useful in our compactness argument. The first is stated in terms of $\chi$ defined in \eqref{ChiFun}.
\begin{cor}\label{StickPartTrajEnergy}
For almost every $t\ge 0$,
\be\label{DiscreteEnergyEst2}
\sum^N_{i=1}m_i\dot\gamma_i(t)^2\le  
\left[\sum^N_{i=1}m_iv_i^2+\sum^N_{i,j=1}m_im_jW'(x_i-x_j)^2 \right]\chi'(t).
\ee
\end{cor}
\begin{proof}
It can be shown using Jensen's inequality and property $(iv)$, that sticky particle trajectories have nonincreasing energy: for each $0\le s< t$
\begin{align}\label{ContNonincreaseEnergy}
&\frac{1}{2}\sum^N_{i=1}m_i\dot\gamma_i(t+)^2+\frac{1}{2}\sum^N_{i,j=1}m_im_jW(\gamma_i(t)-\gamma_j(t))\\
&\quad \le \frac{1}{2}\sum^N_{i=1}m_i\dot\gamma_i(s+)^2+
\frac{1}{2}\sum^N_{i,j=1}m_im_jW(\gamma_i(s)-\gamma_j(s)).
\end{align}
Using the convexity of $x\mapsto W(x)+(L/2)x^2$, we can argue very similarly to how we did in deriving \eqref{PrimeEnergyEst} to obtain \eqref{DiscreteEnergyEst2}.  
\end{proof}

\par The final set of estimates quantify how particles stick together and are stated as follows.  These bounds are proved Proposition 3.4 and 3.8 in \cite{Hynd}, so we will omit the required argument here. 
\begin{prop}\label{PropQSPP}
Fix $i,j\in \{1,\dots,N\}$. \\ (a) For $0<s\le t$,
\be\label{QSPP}
\frac{|\gamma_i(t)-\gamma_j(t)|}{\sinh(\sqrt{L}t)}
\le \frac{|\gamma_i(s)-\gamma_j(s)|}{\sinh(\sqrt{L}s)}.
\ee 
(b) For $x_i\ge x_j$ and $t\ge 0$,  
\be\label{timeZeroEst}
\gamma_i(t)-\gamma_j(t)\le \cosh(\sqrt{L}t)(x_i-x_j)+\frac{1}{\sqrt{L}}\sinh(\sqrt{L}t)\int^{x_i}_{x_j}|v_0'(x)|dx.
\ee
\end{prop}

\par We can use the trajectories to build a solution to  \eqref{FlowMapEqn} as follows. Define 
\be\label{DiscreteXX}
X(t): \{x_1,\dots, x_N\}\rightarrow \R; x_i\mapsto \gamma_i(t)
\ee
for $t\ge 0$. It is clear that $X(t)\in L^2(\rho_0)$ for each $t\ge 0$. 

\begin{prop}\label{DiscreteSolnProp}
The mapping $X$ defined in \eqref{DiscreteXX} has the following properties. 

\begin{enumerate}[(i)]

\item $X(0)=\textup{id}_\R$ and 
\be
\dot X(t)=\E_{\rho_0}\left[v_0-\displaystyle\int^t_0(W'*\rho_s)(X(s))ds\bigg| X(t)\right]
\ee
for all but finitely many $t\ge 0$. Both equalities hold on the support of $\rho_0$. 

\item $E(t)\le E(s)$, for $s\le t$. Here 
\begin{align*}
E(\tau):=\int_{\R}\frac{1}{2}\dot X(\tau+)^2d\rho_0+\int_{\R}\int_{\R}\frac{1}{2}W(X(y,\tau)-X(z,\tau))d\rho_0(y)d\rho_0(z).
\end{align*}

\item $X: [0,\infty)\rightarrow L^2(\rho_0); t\mapsto X(t)$ is locally Lipschitz continuous.

\item For $t\ge 0$ and $y,z\in\textup{supp}(\rho_0)$ with $y\le z$, 
$$
0\le X(z,t)-X(y,t)\le \cosh(\sqrt{L}t)(z-y)+\frac{1}{\sqrt{L}}\sinh(\sqrt{L}t)\int^z_y|v_0'(x)|dx.
$$

\item For each $0<s\le t$ and $y,z\in\textup{supp}(\rho_0)$
$$
\frac{|X(y,t)-X(z,t)|}{\sinh(\sqrt{L}t)}\le \frac{|X(y,s)-X(z,s)|}{\sinh(\sqrt{L}s)}.
$$

\item There is a Borel $v: \R\times[0,\infty)\rightarrow \R$ such that 
$$
\dot X(t)=v(X(t),t)
$$
for almost every $t\ge 0$. 

\end{enumerate}

\end{prop}
\begin{proof}
 $(i)$ follows from Corollary \ref{AveProp}.  $(ii)$ is due to inequality \eqref{ContNonincreaseEnergy}. In view of \eqref{DiscreteEnergyEst2}, 
 \begin{align}
 \int_{\R}(X(t)-X(s))^2\rho_0&\le (t-s)^{1/2}\int^t_s\left(\int_{\R}\dot X(\tau)^2\rho_0\right)d\tau\\
 &\le\left[\sum^N_{i=1}m_iv_i^2+\sum^N_{i,j=1}m_im_jW'(x_i-x_j)^2 \right](t-s)^{1/2}  \int^t_s\chi(\tau)d\tau 
 \end{align}
 for $0\le s\le t$. We conclude $(iii)$. $(iv)$ and $(v)$ follow from Proposition \ref{PropQSPP}. As for $(vi)$, set
 $$
v(x,t):=
\begin{cases}
\dot\gamma_i(t+), \quad &x=\gamma_i(t)\\
0, \quad &\text{otherwise}
\end{cases}
 $$
for $x\in \R$ and $t\ge 0$. This function is well defined by part $(iii)$ of Proposition \ref{StickyParticlesExist}, and it is routine to check that $v$ is Borel. By the definition of $X$, 
we see $\dot X(t)=v(X(t),t)$ for almost every $t\ge 0$. 
\end{proof}

\subsection{Convergence}
Now we will suppose $\rho_0\in {\cal P}(\R)$ is a general probability measure satisfying \eqref{InitRhoZeroVeeZero}.  We can select a sequence 
$(\rho^k_0)_{k\in \N}\subset {\cal P}(\R)$ for which each $\rho^k_0$ is of the form \eqref{ConvCombDirac} and
$$
\lim_{k\rightarrow\infty}\int_{\R}gd\rho^k_0=\int_{\R}gd\rho_0
$$
for continuous $g: \R\rightarrow \R$ which grows at most quadratically.  By Proposition \ref{DiscreteSolnProp}, there is an absolutely continuous 
$X^k:  [0,\infty)\rightarrow L^2(\rho^k_0)$ for which satisfies $(i)-(vi)$ of that assertion with $\rho^k_0$ replacing $\rho_0$ for each $k\in \N$. 

\par It turns out that $(X^k)_{k\in \N}$ is compact in a certain sense. In particular, we established the following claim in section in section 4 of  \cite{Hynd}.

\begin{prop}\label{StrongCompactXkayjay}
There is a subsequence $(X^{k_j})_{j\in \N}$ and a locally Lipschitz $X: [0,\infty)\rightarrow L^2(\rho_0)$ such that 
\be\label{StrongConvLimXkayJay}
\lim_{j\rightarrow\infty}\int_{\R}h(\textup{id}_\R,X^{k_j}(t))d\rho^{k_j}_0=\int_{\R}h(\textup{id}_\R,X(t))d\rho_0
\ee
for each $t\ge 0$ and continuous $h:\R^2\rightarrow \R$ with
\be
\sup_{(x,y)\in \R^2}\frac{|h(x,y)|}{1+x^2+y^2}<\infty.
\ee
Furthermore, $X$ has the following properties.  

\begin{enumerate}[(i)]

\item $X$ is a solution of \eqref{FlowMapEqn} which satisfies \eqref{xInit}.

\item 
$E(t)\le E(s)$, for Lebesgue almost every $s\le t$, where 
\begin{align*}
E(\tau):=\int_{\R}\frac{1}{2}\dot X(\tau+)^2d\rho_0+\int_{\R}\int_{\R}\frac{1}{2}W(X(y,\tau)-X(z,\tau))d\rho_0(y)d\rho_0(z).
\end{align*}

\item For $t\ge 0$ and $y,z\in\textup{supp}(\rho_0)$ with $y\le z$, 
$$
0\le X(z,t)-X(y,t)\le \cosh(\sqrt{L}t)(z-y)+\frac{1}{\sqrt{L}}\sinh(\sqrt{L}t)\int^z_y|v_0'(x)|dx.
$$

\item For each $0<s\le t$ and $y,z\in\textup{supp}(\rho_0)$
$$
\frac{|X(y,t)-X(z,t)|}{\sinh(\sqrt{L}t)}\le \frac{|X(y,s)-X(z,s)|}{\sinh(\sqrt{L}s)}.
$$

\item There is a Borel $v: \R\times[0,\infty)\rightarrow \R$ such that 
$$
\dot X(t)=v(X(t),t)
$$
for almost every $t\ge 0$. 
\end{enumerate}
\end{prop}

\par We can now sketch a proof of Theorem  \ref{EPthm}. 

\begin{proof}[Proof of Theorem \ref{EPthm}].
Suppose $X$ is the map obtained in Proposition \ref{StrongCompactXkayjay} above and set
$$
\rho: [0,\infty)\rightarrow {\cal P}(\R); t\mapsto X(t){_\#}\rho.
$$
By Lemma \ref{solnPElemma}, $\rho$ and $v$ from part $(v)$ is a weak solution pair of \eqref{PressLessEuler} which satisfies the initial conditions \eqref{Init}. As
$$
E(\tau)=\int_{\R}\frac{1}{2}v(x,\tau)^2d\rho_\tau(x)+\iint_{\R^2}\frac{1}{2}W(x-y)d\rho_\tau(x)d\rho_\tau(y)
$$
for almost every $\tau\ge 0$, $E$ is an essentially nonincreasing function by part $(ii)$. 

\par We are only left to verify the entropy inequality \eqref{entropy}.  By part $(vi)$, 
\begin{align*}
0&\ge\frac{d}{dt} \frac{(X(y,t)-X(z,t))^2}{\sinh(\sqrt{L}t)^2}\\
&=\frac{2}{\sinh(\sqrt{L}t)^2}\left( (v(X(y,t),t)-v(X(z,t),t))(X(y,t)-X(z,t))  -\frac{(X(y,t)-X(z,t))^2}{\tanh(\sqrt{L}t)}\right)
\end{align*} 
for Lebesgue almost every $t>0$ and $y,z$ belonging to some Borel $Q\subset \R$  with $\rho_0(Q)=1$. 
That is, \eqref{entropy} holds for $x,y$ belonging to 
$$
X(t)(Q):=\{X(y,t)\in  \R: y\in Q\}
$$
and Lebesgue almost every $t>0$. By part $(iii)$ of the above proposition, $y\mapsto X(y,t)$ is continuous and monotone on $\R$.  As a result, 
 $X(t)(Q)$ is Borel measurable which enables us to check 
 $$
 \rho_t(X(t)(Q))=\rho_0(X(t)^{-1}[X(t)(Q)])\ge \rho_0(Q)=1.
 $$
As a result, we conclude that  \eqref{entropy} holds for for $\rho_t$ almost every $x,y\in \R$ and Lebesgue almost every $t>0$. 
 \end{proof}

\section{Elastodynamics}\label{ElastoDyn}
We will now discuss the the following initial value problem which arises in the dynamics of elastic bodies.
Let us suppose $U\subset\R^d$ is a bounded domain with smooth boundary and $F\in C^1(\Md)$, where $\Md$ is the space of real $d\times d$ matrices. We consider the following initial value problem
\be\label{IVP}
\begin{cases}
\bu_{tt}=\text{div}DF(D\bu)\; &\text{in}\; U\times(0,T)\\
\hspace{.1in}\bu=\bzero&\text{on}\; \partial U\times[0,T)\\
\hspace{.1in}\bu=\bg&\text{on}\; U\times\{0\}\\
\hspace{.05in}\bu_t=\bh&\text{on}\; U\times\{0\}.
\end{cases}
\ee
The unknown is a mapping $\bu:U\times[0,T)\rightarrow\R^d$ and $\bg,\bh:U\rightarrow\R^d$ are given.   Also $\bzero:U\rightarrow \R^d$ is the 
mapping that is identically equal to $0\in \R^d$. 

\par We will use the matrix norm $|A|:=\text{trace}(A^tA)^{1/2}$ and suppose
\be
\Md\ni A\mapsto F(A)+\frac{L}{2}|A|^2
\ee
is convex for some $L> 0$; this semiconvexity assumption is known as the Andrews-Ball condition \cite{MR657784}. Moreover, we will assume that $F$ is coercive. That is,  
\be\label{Fcoercive}
c(|A|^2-1)\le F(A)\le \frac{1}{c}(|A|^2+1)
\ee
for all $A\in \Md$ and some $c\in (0,1]$.  It is not hard to verify that these assumptions together imply that $DF$ grows at most linearly
\be\label{LinearBoundDF}
|DF(A)|\le C(|A|+1), \quad A\in \Md.
\ee
Here $C$ depends on the above constants $L$ and $c$. 

There is a natural conservation law associated with solution of the initial value problem \eqref{IVP}: if $\bu:U\times[0,T)\rightarrow\R^d$ is a smooth solution, then 
\begin{align*}
\frac{d}{dt}\int_U \frac{1}{2}|\bu_t|^2+F(D\bu) dx&=\int_U \bu_t\cdot \bu_{tt}+DF(D\bu)\cdot D\bu_t dx\\
&=\int_U \bu_t\cdot \left(\bu_{tt}-\Div DF(D\bu)\right) dx\\
&=0.
\end{align*}
As a result, 
$$
\int_U \frac{1}{2}|\bu_t(x,t)|^2+F(D\bu(x,t)) dx=\int_U \frac{1}{2}|\bh(x)|^2+F(D\bg(x)) dx
$$
for each $t\in[0,T]$. These observations motivate the definition of a weak solution of the initial value problem \eqref{IVP} below.

\par The following definition makes use of space $H^1_0(U;\R^d)$, which we recall is the closure of the smooth test mappings $\bw\in C^\infty_c(U; \R^d)$ in the Sobolev space
$$
H^1(U;\R^d):=\left\{\bw\in L^2(U;\R^d): \bw\;\text{weakly differentiable with}\;\int_U|D\bw|^2dx<\infty\right\}.
$$
In particular, $H^1_0(U;\R^d)\subset H^1(U;\R^d)$ is naturally the subset of functions which ``vanish" on $\partial U$. We also note that $H^1_0(U;\R^d)$ is a Banach space under the norm 
$$
\|\bw\|_{H^1_0(U;\R^d)}=\left(\int_U|D\bw|^2dx\right)^{1/2}
$$
and its continuous dual space is denoted $H^{-1}(U;\R^d)$. We refer to Chapter 5 of \cite{MR2597943} for more on the theory of Sobolev spaces.

\begin{defn} Suppose $\bg\in H^1_0(U;\R^d)$ and $\bh\in L^2(U;\R^d)$.  A measurable mapping $\bu:U\times [0,T)\rightarrow \R^d$ is a {\it weak solution} for the initial value problem \eqref{IVP} provided 
\be\label{Spaces}
\bu\in L^\infty([0,T]; H^1_0(U;\R^d))\quad \text{and}\quad \bu_t\in L^\infty([0,T]; L^2(U;\R^d)),
\ee
\begin{align*}
\int^T_0\int_U \phi_t\cdot \bu_tdxdt+\int_U \phi|_{t=0}\cdot \bh dx=\int^T_0\int_U D\phi\cdot DF(D\bu)dxdt
\end{align*}
for each $\phi\in C^\infty_c(U\times[0,T);\R^d)$, and 
\be\label{InitG}
\bu(\cdot,0)=\bg.
\ee
\end{defn}
\begin{rem}
We note that the initial condition \eqref{InitG} can be set as \eqref{Spaces} implies that $\bu:[0,T)\rightarrow L^2(U;\R^d)$ can be identified with a continuous map. 
\end{rem}
\begin{rem}\label{WeakContRem}
Weak solutions also satisfy
\begin{align*}
\int_U \bw \cdot \bu_t(\cdot,t)dx=\int_U \bw\cdot \bh dx-\int^t_0\int_U D\bw(x)\cdot DF(D\bu(x,\tau))dxd\tau
\end{align*}
for each $t\in[0,T]$ and $\bw\in H^1_0(U;\R^d)$. In particular, 
\be\label{SimpleHminues1bound}
\|\bu_{tt}(\cdot,t)\|_{H^{-1}(U;\R^d)}=\|DF(D\bu(\cdot,t))\|_{L^2(U;\R^d)}
\ee
for almost every $t\in (0,T)$. 
\end{rem}

Unfortunately, it is unknown whether or not weak solutions of the initial value problem \eqref{IVP} exist. So we will work with an alternative notion of solution. Recall that $\Md$ is a complete, separable metric space under the distance $(A_1,A_2)\mapsto |A_1-A_2|$. Therefore, we can consider ${\cal P}(\Md)$ the collection Borel probability measures on this space.

\begin{defn} Suppose $\bg\in H^1_0(U;\R^d)$ and $\bh\in L^2(U;\R^d)$.  A {\it Young measure solution} of the initial value problem \eqref{IVP} 
is a pair $(\bu,\gamma)$, where $\bu:U\times [0,T)\rightarrow \R^d$ is measurable mapping satisfying \eqref{Spaces} and $\gamma=(\gamma_{x,t})_{(x,t)\in U\times [0,T]}$ is a family of Borel probability measures on $\Md$ which satisfy:
\be\label{YMWeakSolnCond}
\int^T_0\int_U \phi_t\cdot \bu_tdxdt+\int_U \phi|_{t=0}\cdot \bh dx=\int^T_0\int_U D\phi\cdot \left(\int_{\Mnd}DF(A)d\gamma_{x,t}(A)\right)dxdt
\ee
for each $\phi\in C^\infty_c(U\times[0,T);\R^d)$, 
\be\label{GradientYMcond}
D\bu(x,t)=\int_{\Mnd}Ad\gamma_{x,t}(A)
\ee
for Lebesgue almost every $(x,t)\in U\times[0,T]$, and 
\be\label{InitG}
\bu(\cdot,0)=\bg.
\ee
\end{defn}
\begin{rem}
Of course if $\gamma_{x,t}=\delta_{D\bu(x,t)}$ for almost every $(x,t)$, then the $\bu$ is a weak solution. 
\end{rem}
We note that Demoulini first verified existence of Young measure solution in \cite{MR1437192} by using an implicit time scheme related to the  
initial value problem \eqref{IVP}. Some other notable works on this existence problem are \cite{MR2075752, MR3093380, MR3482395, MR2000976}.
We specifically mention Rieger's paper \cite{MR2000976} as it reestablished Demoulini's result using an approximation scheme involving 
an initial value problem \eqref{IVPmu} which we will study below.  Alternatively, we will pursue existence via Galerkin's method. 

\subsection{Galerkin's method}
Let $\{\phi_j\}_{j\in \N}\subset L^2(U)$ denote the the orthonormal basis of eigenfunctions of the Laplace operator on $U$ with Dirichlet boundary conditions. That is, 
$$
\begin{cases}
-\Delta\phi_j=\lambda_j\phi_j\quad & \text{in}\; U\\
\hspace{.27in}\phi_j=0\quad & \text{on}\; \partial U
\end{cases}
$$
where the sequence of eigenvalues $\{\lambda_i\}_{i\in \N}$ are positive, nondecreasing and tend to $\infty$.  We also have 
$$
\int_U D\phi_j\cdot D\phi_kdx=\int_U (-\Delta\phi_j) \phi_kdx=\lambda_j\int_U \phi_j \phi_kdx=\lambda_j\delta_{jk}.
$$
It follows that $\{\phi_j\}_{j\in \N}\subset H^1_0(U)$ is an orthogonal basis.  With these functions, we can set
$$
\bg^N:=\sum^N_{j=1}\bg_j\phi_j\quad \text{and}\quad \bh^N:=\sum^N_{j=1}\bh_j\phi_j
$$
for each $N\in \N$. Here $\bg_j=\int_U\bg \phi_jdx$ and $\bh_j=\int_U\bh \phi_jdx$. In particular, notice that $\bg^N\rightarrow \bg$ in $H^1_0(U;\R^d)$ and $\bh^N\rightarrow \bh$ in $L^2(U;\R^d)$ as $N\rightarrow\infty$. 

\par We will now use Proposition \ref{ODENewtonProp} to generate an approximation sequence to the Young measure solution of \eqref{IVP} that we seek.  
\begin{lem}
For each $N\in \N$, there is a weak solution $\bu^N$ of \eqref{IVP} with $\bg^N$ and $\bh^N$ replacing $\bg$ and $\bh$, respectively. 
\end{lem}
\begin{proof}
It suffices to find a weak solution of the form 
$$
\bu^N(x,t)=\sum^N_{j=1}\ba_j(t)\phi_j(x)
$$
$x\in U$, $t\in [0,T]$ for appropriate mappings $\ba_j:[0,T]\rightarrow \R^d$ $(j=1,\dots, N)$. In particular, this ansatz is a weak solution if and only if 
\be\label{ODEVeeM}
\begin{cases}
\ddot \ba_j(t)=-D_{y_j}V(\ba_1(t),\dots, \ba_N(t)),\quad t\in (0,T)\\
\ba_j(0)=\bg_j\\
\dot \ba_j(0)=\bh_j
\end{cases}
\ee
$(j=1,\dots, N)$.  Here 
\be\label{discreteVee}
V(y_1,\dots, y_N):=\int_U F\left(\sum^N_{i=1}y_j\otimes D\phi_j(x)\right)dx
\ee
for $y_1,\dots,y_N\in \R^d$.

\par Since $\int_U D\phi_j\cdot D\phi_kdx=\lambda_j\delta_{jk}$, we may compute 
$$
\int_U\left|\sum^m_{j=1}y_j\otimes D\phi_j(x)\right|^2dx=\sum^m_{j=1}\lambda_j|y_j|^2.
$$
Moreover, we have that
\begin{align*}
&V(y_1,\dots, y_m)+\frac{L\lambda_N}{2}\sum^N_{j=1}|y_j|^2\\
&=\int_U F\left(\sum^N_{i=1}y_j\otimes D\phi_j(x)\right) +\frac{L}{2}\left|\sum^N_{j=1}y_j\otimes D\phi_j(x)\right|^2dx +\frac{L}{2}\sum^N_{j=1}(\lambda_N-\lambda_j)|y_j|^2
\end{align*}
is convex. Thus, $V: (\R^d)^N\rightarrow \R$ is semiconvex. It then follows from Proposition \ref{ODENewtonProp} that a solution $\ba_1,\dots, \ba_m\in C^2([0,T];\R^d)$ of the multidimensional ODE system \eqref{ODEVeeM} exists. 
\end{proof}

\par As the $\phi_j$ are each smooth on $U$, $\bu^N$ is classical solution of the IVP \eqref{IVP}.  By the conservation of energy we then have 
\be\label{ConservationEnergyM}
\int_U \frac{1}{2}|\bu^N_t(x,t)|^2+F(D\bu^N(x,t)) dx=\int_U \frac{1}{2}|\bh^N(x)|^2+F(D\bg^N(x)) dx
\ee
for $t\in [0,T]$. Since the right hand side above is bounded uniformly in $N\in \N$, $(\bu^N)_{N\in \N}$ is bounded in the space determined by \eqref{Spaces}.  We now assert that $(\bu^N)_{N\in \N}$ has a subsequence that converges in various senses to a mapping $\bu$ which satisfies \eqref{Spaces}. 

\begin{lem}\label{umkCompactnessLem}
There is a subsequence $(\bu^{N_k})_{k\in \N}$ and a measurable mapping $\bu$ which satisfies \eqref{Spaces} for which
$$
\begin{cases}
\bu^{N_k}\rightarrow \bu\; \text{in}\; C([0,T];L^2(U;\R^d))\\\\
\bu^{N_k}(\cdot,t)\rightharpoonup \bu(\cdot, t)\; \text{in}\; H^1_0(U;\R^d)\;\text{for all}\; t\in [0,T]\\\\
\bu^{N_k}_t\rightarrow \bu_t\; \text{in}\; C([0,T];H^{-1}(U;\R^d))\\\\
\bu^{N_k}_t(\cdot,t)\rightharpoonup \bu_t(\cdot,t)\; \text{in}\; L^2(U;\R^d)\;\text{for all}\; t\in [0,T].
\end{cases}
$$
as $k\rightarrow\infty$.  Moreover, $\bu(\cdot,0)=\bg$ and $\bu_t(\cdot,0)=\bh$.
\end{lem} 
\begin{proof}
By the \eqref{Fcoercive} and \eqref{ConservationEnergyM}, 
\be\label{UnifBoundNormsM}
\sup_{N\in \N}\sup_{0\le t\le T}\int_U|\bu^N_t(x,t)|^2 +|D\bu^N(x,t)|^2dx<\infty.
\ee
It follows that the sequence of functions $\bu^N:[0,T]\rightarrow L^2(U;\R^d)$ $(N\in \N)$ is equicontinuous and the sequence $(\bu^N(\cdot,t))_{N\in\N}\subset L^2(U;\R^d)$ is precompact for each $t\in [0,T]$. By the Arzel\`a-Ascoli theorem, there is a subsequence $(\bu^{N_k})_{k\in \N}$ that converges uniformly to a continuous $\bu:[0,T] \rightarrow L^2(U;\R^d).$   For $t\in (0,T]$, $(\bu^{N_k}(\cdot,t))_{k\in\N}\subset H^1_0(U;\R^d)$ is bounded and thus has a weakly convergent subsequence. However, any weak limit of this sequence must be equal to $\bu(\cdot, t)$, and so the entire sequence must converge weakly in $H^1_0(U;\R^d)$ to this mapping. As a result, 
$$
\bu^{N_k}(\cdot,t)\rightharpoonup \bu(\cdot, t)\; \text{in}\; H^1_0(U;\R^d)
$$
for each $ t\in [0,T]$.

\par In view of \eqref{SimpleHminues1bound}
$$
\|\bu^N_{tt}(\cdot,t)\|_{H^{-1}(U;\R^d)}=\|DF(D\bu^N(\cdot,t))\|_{L^2(U;\R^d)}.
$$
Combining this with \eqref{LinearBoundDF} and \eqref{UnifBoundNormsM}, we see that $\bu^N_t:[0,T]\rightarrow H^{-1}(U;\R^d)$ is equicontinuous. The uniform bound \eqref{UnifBoundNormsM} also implies this sequence is pointwise bounded in $L^2(U;\R^d)$ which is a compact subspace of $H^{-1}(U;\R^d)$. As a result, there is a subsequence (not relabeled)  such that $\bu^{N_k}_t(\cdot,t)\rightarrow \bu_t(\cdot,t)$ in $H^{-1}(U;\R^d)$ uniformly in $t\in [0,T]$. What's more,  as
this sequence is pointwise uniformly bounded in $L^2(U;\R^d)$, $\bu^{N_k}_t(\cdot,t)\rightharpoonup\bu_t(\cdot,t)$ in $L^2(U;\R^d)$ for each $t\in [0,T]$.  It is then clear that $\bu(\cdot, 0)=\bg$ and $\bu_t(\cdot, 0)=\bh$ by weak convergence; and by \eqref{UnifBoundNormsM}, $\bu$ also satisfies \eqref{Spaces}.
\end{proof}

\par We are now  ready to verify the existence of a Young measure solution.  We will use the above compactness assertion and \eqref{Prokhorov}.
\begin{thm}
For each $\bg\in H^1_0(U;\R^d)$ and $\bh\in L^2(U;\R^d)$, there exists a Young measure solution pair $(\bu,\gamma)$. 
\end{thm}
\begin{proof}
First, define the sequence of Borel probability measures $(\eta^N)_{N\in \N}\in {\cal P}(\overline{U}\times[0,T]\times \Md)$ via
$$
\int^T_0\int_{\overline{U}}\int_{\Md}g(x,t,A)d\eta^N(x,t,A):=
\frac{1}{T|U|}\int^T_0\int_{U}g(x,t,D\bu^N(x,t))dxdt
$$
for each $g\in C_b(\overline{U}\times[0,T]\times \Md)$.  Since $F$ is coercive \eqref{Fcoercive},  
\begin{align*}
\int^T_0\int_{\overline{U}}\int_{\Md}|A|^2d\eta^N(x,t,A)&=\frac{1}{T|U|}\int^T_0\int_{U}|D\bu|^2dxdt\\
&\le \frac{1}{T|U|} \int^T_0\int_{U}1+\frac{1}{c}F(D\bu^N)dxdt\\
&= 1+\frac{1}{cT|U|}\int^T_0\int_{U}F(D\bu^N)dxdt\\
&\le  1 +\frac{1}{c|U|} \int_U \frac{1}{2}|\bh^N|^2+F(D\bg^N) dx\\
&\le  1 +\frac{1}{c|U|} \int_U \frac{1}{2}|\bh^N|^2+c\left(|D\bg^N|^2+1\right) dx
\end{align*}
which is bounded uniformly in $N$. 

\par As the function 
$$
\overline{U}\times[0,T]\times \Md\rightarrow \R; (x,t,A)\mapsto |A|^2
$$
has compact sublevel sets, the sequence $(\eta^N)_{N\in\N}$ is tight. By \eqref{Prokhorov}, this sequence has a narrowly convergent subsequence which we will label $(\eta^{N_k})_{k\in\N}$. In particular, there is $\eta^\infty\in {\cal P}(\overline{U}\times[0,T]\times \Md)$ such that 
\be\label{narrowConvofEetaM}
\lim_{k\rightarrow\infty}\int^T_0\int_{\overline{U}}\int_{\Md}g(x,t,A)d\eta^{N_k}(x,t,A)=
\int^T_0\int_{\overline{U}}\int_{\Md}g(x,t,A)d\eta^\infty(x,t,A)
\ee
for each $g\in C_b(\overline{U}\times[0,T]\times \Md)$.  Moreover, as 
$$
\sup_N\int^T_0\int_{\overline{U}}\int_{\Md}|A|^2d\eta^N(x,t,A)<\infty,
$$
the limit \eqref{narrowConvofEetaM} actually holds for each continuous $g$ which satisfies $|g(x,t,A)|\le C(|A|+1)$.

\par The weak solution condition equation for $\bu^N$ may be written
\begin{align*}
\int^T_0\int_U \phi_t\cdot \bu^N_tdxdt+\int_U \phi|_{t=0}\cdot \bh^N dx&=\int^T_0\int_U D\phi\cdot DF(D\bu^N)dxdt\\
&=T|U|\int^T_0\int_{\overline{U}}\int_{\Md} D\phi\cdot DF(A) d\eta^N(x,t,A).
\end{align*}
Here $\phi\in C_c^\infty(U\times[0,T);\R^d)$. As $DF: \Md\rightarrow \Md$ continuous and grows linearly, we can pass to the limit as $N_k\rightarrow\infty$ to get 
\be\label{almostYMWeakSolnCond}
\int^T_0\int_U \phi_t\cdot \bu_tdxdt+\int_U \phi|_{t=0}\cdot \bh dx=T|U|\int^T_0\int_{\overline{U}}\int_{\Md} D\phi\cdot DF(A) d\eta^\infty(x,t,A).
\ee
Here we made use of the convergence detailed in Lemma \ref{umkCompactnessLem}.

\par We also have
$$
\int^T_0\int_{\overline{U}}\int_{\Md}g(x,t)d\eta^\infty(x,t,A):=
\frac{1}{T|U|}\int^T_0\int_{U}g(x,t)dxdt
$$
for each $g\in C_b(\overline{U}\times[0,T])$. By the disintegration of probability measures (Theorem 5.3.1 of \cite{AGS}), there is a family of Borel probability measures 
$$
(\gamma_{x,t})_{(x,t)\in U\times(0,T)}\subset  {\cal P}(\Md)
$$
such that 
$$
\int^T_0\int_{\overline{U}}\int_{\Md}g(x,t,A)d\eta^\infty(x,t,A)=\frac{1}{T|U|}\int^T_0\int_{U}\left(\int_{\Md}g(x,t,A)d\gamma_{x,t}(A)\right)dxdt.
$$
Combining this observation with \eqref{almostYMWeakSolnCond}, we conclude that condition \eqref{YMWeakSolnCond} is satisfied.

\par Also recall that $(D\bu^{N_k})_{k\in \N}$ converges weakly to $D\bu$ in $L^2(U\times[0,T];\Md)$. It then follows that for each continuous $G:\overline{U}\times[0,T]\rightarrow \R$,
\begin{align*}
\int^T_0\int_{U}G(x,t)\cdot \left(\int_{\Md}Ad\gamma_{x,t}(A)\right)dxdt&=\int^T_0\int_{\overline{U}}\int_{\Md}G(x,t)\cdot A\; d\eta^\infty(x,t,A)\\
&=\lim_{k\rightarrow\infty}\int^T_0\int_{\overline{U}}\int_{\Md}G(x,t)\cdot A\; d\eta^{N_k}(x,t,A)\\
&=\lim_{k\rightarrow\infty}\int^T_0\int_{U}G\cdot D\bu^{N_k}dxdt\\
&=\int^T_0\int_{U}G\cdot D\bu dxdt.
\end{align*}
That is, \eqref{GradientYMcond} holds. We finally recall that Lemma \ref{umkCompactnessLem} gives that $\bu(\cdot, 0)=\bg$ and conclude that the pair $(\bu,\gamma)$ is a Young measure solution of \eqref{IVP}.
\end{proof}

\subsection{Analysis of damped model}
The last initial value problem that we will consider is
\be\label{IVPmu}
\begin{cases}
\bu_{tt}=\text{div}DF(D\bu)+\mu \Delta \bu_t\; &\text{in}\; U\times(0,T)\\
\hspace{.1in}\bu=\bzero&\text{on}\; \partial U\times[0,T)\\
\hspace{.1in}\bu=\bg&\text{on}\; U\times\{0\}\\
\hspace{.05in}\bu_t=\bh&\text{on}\; U\times\{0\}.
\end{cases}
\ee
Here $\mu>0$ is known as a damping parameter as the energy of smooth solutions dissipates
\be
\frac{d}{dt}\int_U \frac{1}{2}|\bu_t|^2+F(D\bu) dx=-\mu\int_U|D\bu_t|^2dx.
\ee
In particular,   
\be
\int_U \frac{1}{2}|\bu_t(x,t)|^2+F(D\bu(x,t)) dx+\mu\int^t_0\int_U|D\bu_t|^2dxds=\int_U \frac{1}{2}|\bh|^2+F(D\bg)dx
\ee
for $t\ge 0$.  Weak solutions of \eqref{IVPmu} are defined as follows. 
\begin{defn} Suppose $\bg\in H^1_0(U;\R^d)$ and $\bh\in L^2(U;\R^d)$.  A measurable mapping $\bu:U\times [0,T)\rightarrow \R^d$ is a {\it weak solution} for the initial value problem \eqref{IVPmu} provided 
\be\label{SpacesMu}
\bu\in L^\infty([0,T]; H^1_0(U;\R^d))\quad \text{and}\quad  \bu_t\in L^\infty([0,T]; L^2(U;\R^d))\cap L^2([0,T]; H^1_0(U;\R^d)).
\ee
\be\label{WeadSolnDamped}
\int^T_0\int_U \phi_t\cdot \bu_tdxdt+\int_U \phi|_{t=0}\cdot \bh dx=\int^T_0\int_U D\phi\cdot\left( DF(D\bu)+\mu D\bu_t\right)dxdt
\ee
for each $\phi\in C^\infty_c(U\times[0,T);\R^d)$, and 
\be\label{InitG}
\bu(\cdot,0)=\bg.
\ee
\end{defn}

\par We note that Andrews and Ball verified existence of a weak solution when $d=1$ \cite{MR657784}. This result was generalized to all 
$d>1$ by Dolzmann and Friesecke \cite{MR1434040}. Dolzmann and Friesecke used an implicit time scheme and wondered if a Galerkin 
type method could be used instead. Later, Feireisl and Petzeltov\'a showed that this can be accomplished \cite{MR1873012}; see also the following recent papers \cite{MR3093380, MR2421043,  MR1169897} which involved similar existence problems and results.

\par We will also use Galerkin's method and give a streamlined proof of the existence of a weak solution for given initial conditions. This involves finding a weak solution $\bu^N$ with initial conditions $\bu^N(\cdot, 0)=\bg^N$ and $\bu^N_t(\cdot, 0)=\bh^N$ for each $N\in \N$.  An easy computation shows that it is enough to find a solution of the form
$$
\bu^N(x,t)=\sum^N_{j=1}\ba_j(t)\phi_j(x).
$$
where the $\ba_1,\dots, \ba_N$ satisfy 
\be
\begin{cases}
\ddot \ba_j(t)=-D_{y_j}V(\ba_1(t),\dots, \ba_m(t)) -\mu\lambda_j\dot \ba_j(t),\quad t\in (0,T)\\
\ba_j(0)=\bg_j\\
\dot \ba_j(0)=\bh_j
\end{cases}
\ee
for $j=1,\dots, N$.   Here $V$ is defined as in \eqref{discreteVee} and the corresponding energy of this ODE system is  
$$
\sum^N_{j=1}\frac{1}{2}|\dot \ba_j(t)|^2 + V(\ba_1(t),\dots, \ba_m(t))+\mu\int^t_0\sum^m_{j=1}\lambda_j|\dot \ba_j(s)|^2ds 
=\sum^N_{j=1}\frac{1}{2}|\bh_j|^2 + V(\bg_1,\dots, \bg_m)
$$

\par A minor variation of the method we used to verify existence for Newton systems in Proposition \ref{ODENewtonProp} can be used to show that the
above system has a solution $\ba_1,\dots, \ba_N\in C^2([0,T];\R^d)$.  We leave the details to the reader. The weak solution $\bu^N$ obtained
 is a classical solution and so
\be\label{ConsUmMu}
\int_U \frac{1}{2}|\bu^N_t(x,t)|^2+F(D\bu^N(x,t)) dx+\mu\int^t_0\int_U|D\bu_t^N|^2dxds= \int_U \frac{1}{2}|\bh^N|^2+F(D\bg^N) dx
\ee
for $t\in [0,T]$.  We will now focus on using the extra gain in the energy to verify the existence of a weak solution.

\begin{lem}\label{LastCompactUNlemma}
There is a subsequence $(\bu^{N_k})_{k\in \N}$ and a measurable mapping $\bu$ which satisfies \eqref{SpacesMu} for which
$$
\begin{cases}
\bu^{N_k}\rightarrow \bu\; \text{in}\; C([0,T];L^2(U;\R^d))\\\\
\bu^{N_k}(\cdot,t)\rightharpoonup \bu(\cdot, t)\; \text{in}\; H^1_0(U;\R^d)\;\text{for all}\; t\in [0,T]\\\\
\bu^{N_k}_t\rightarrow \bu_t\; \text{in}\; C([0,T];H^{-1}(U;\R^d))\\\\
\bu^{N_k}_t(\cdot,t)\rightharpoonup \bu_t(\cdot,t)\; \text{in}\; L^2(U;\R^d)\;\text{for all}\; t\in [0,T]\\\\
\bu^{N_k}_t\rightarrow \bu_t\; \text{in}\; L^2(U\times[0,T];\R^d)\\\\
\bu^{N_k}\rightarrow \bu\; \text{in}\; L^2([0,T];H^1_0(U;\R^d))
\end{cases}
$$
as $k\rightarrow\infty$.  Moreover, $\bu(\cdot,0)=\bg$ and $\bu_t(\cdot,0)=\bh$.
\end{lem}
\begin{proof}
By our coercivity assumptions on $F$, there is a constant $C$ such that
\be\label{UnifBoundNormsMmu}
\sup_{0\le t\le T}\int_U|\bu^N_t(x,t)|^2 +|D\bu^N(x,t)|^2dx+\int^T_0\int_U|D\bu_t^N|^2dxds\le C
\ee
for every $N\in\N$.  Arguing as we did in the proof of Lemma \ref{umkCompactnessLem}, the first four assertions in braces hold along with $\bu(\cdot,0)=\bg$ and $\bu_t(\cdot,0)=\bh$ for some $\bu$ that satisfies \eqref{Spaces}. Moreover, $\bu^N_t\in L^2([0,T];H^1_0(U;\R^d))$ is bounded and thus $\bu^{N_k}_t\rightharpoonup \bu_t$ in $L^2([0,T];H^1_0(U;\R^d))$. Consequently, 
$$
\int^T_0\int_U|D\bu_t|^2dxds\le C
$$
and $\bu$ satisfies \eqref{SpacesMu}. 

\par We are left to verify the last two assertion in the braces.  Recall that for each $\delta>0$,  
$$
\|v\|^2_{L^2(U)}\le \frac{1}{2\delta} \|v\|^2_{H^{-1}(U)}+\frac{\delta}{2} \| v\|^2_{H^1_0(U)}
$$
for every $v\in H^1_0(U)$.  It follows that
\begin{align*}
\int^T_0\|\bu^N_t(\cdot,t)-\bu_t(\cdot,t)\|^2_{L^2(U)}dt&\le \frac{1}{2\delta} \int^T_0\|\bu^N_t(\cdot,t)-\bu_t(\cdot,t)\|^2_{H^{-1}(U)}dt\\
&\quad +\frac{\delta}{2}  \int^T_0\|\bu^N_t(\cdot,t)-\bu_t(\cdot,t)\|^2_{H^{1}_0(U)}dt\\
&\le \frac{1}{2\delta}\int^T_0\|\bu^N_t(\cdot,t)-\bu_t(\cdot,t)\|^2_{H^{-1}(U)}dt+2C\delta.
\end{align*}
By the uniform convergence of $(\bu^{N_k}_t)_{k\in\N}$ in $H^{-1}(U;\R^d)$, we have 
$$
\lim_{k\rightarrow\infty}\int^T_0\|\bu^{N_k}_t(\cdot,t)-\bu_t(\cdot,t)\|^2_{L^2(U)}dt\le 2C\delta.
$$
As $\delta>0$ is arbitrary, $\bu^{N_k}_t\rightarrow \bu_t\; \text{in}\; L^2(U\times[0,T];\R^d)$.

\par In order to verify that $D\bu^{N_k}$ converges to $D\bu$ in $L^2(U\times [0,T];\Md)$, we will use the  strong convergence of $\bu^{N_k}_t$ and the fact that $F$ is semiconvex.  The computation below is also reminiscent of the proof of Proposition 3.1 in \cite{MR1434040}.  Let $k,\ell\in \N$ and observe
\begin{align*}
\frac{d}{dt}\int_U\frac{1}{2}|D\bu^{N_k}-D\bu^{N_\ell}|^2dx&=\int_U(D\bu^{N_k}-D\bu^{N_\ell})\cdot (D\bu^{N_k}_t-D\bu^{N_\ell}_t)dx\\
&=-\int_U(\bu^{N_k}-\bu^{N_\ell})\cdot (\Delta\bu^{N_k}_t-\Delta\bu^{N_\ell}_t)dx\\
&=-\int_U(\bu^{N_k}-\bu^{N_\ell})\cdot (\bu^{N_k}_{tt}-\bu^{N_\ell}_{tt}-\Div(DF(D\bu^{N_k})-DF(D\bu^{N_\ell}))dx\\
&=-\int_U(\bu^{N_k}-\bu^{N_\ell})\cdot (\bu^{N_k}_{tt}-\bu^{N_\ell}_{tt})dx \\
&\quad -\int_UDF(D\bu^{N_k})-DF(D\bu^{N_\ell})\cdot (D\bu^{N_k}-D\bu^{N_\ell})dx\\
&\le-\int_U(\bu^{N_k}-\bu^{N_\ell})\cdot (\bu^{N_k}_{tt}-\bu^{N_\ell}_{tt})dx +L\int_U|D\bu^{N_k}-D\bu^{N_\ell}|^2dx.
\end{align*}
Integrating this inequality gives 
\begin{align*}
\int_U\frac{1}{2}|D\bu^{N_k}(x,t)-D\bu^{N_\ell}(x,t)|^2dx&\le \int_U\frac{1}{2}|D\bg^{N_k}-D\bg^{N_\ell}|^2dx + L\int^t_0\int_U|D\bu^{N_k}-D\bu^{N_\ell}|^2dxdt\\
&\quad -\int^t_0\int_U(\bu^{N_k}-\bu^{N_\ell})\cdot (\bu^{N_k}_{tt}-\bu^{N_\ell}_{tt})dxdt\\
&\le \int_U\frac{1}{2}|D\bg^{N_k}-D\bg^{N_\ell}|^2dx + L\int^t_0\int_U|D\bu^{N_k}-D\bu^{N_\ell}|^2dxdt\\
&\quad +\int^t_0\int_U|\bu^{N_k}_{t}-\bu^{N_\ell}_{t}|^2dxdt -\int_U(\bu^{N_k}-\bu^{N_\ell})\cdot (\bu^{N_k}_{t}-\bu^{N_\ell}_{t})dx|^{t=T}_{t=0}\\
&\le o(1)+ L\int^t_0\int_U|D\bu^{N_k}-D\bu^{N_\ell}|^2dxdt 
\end{align*}
as $k,\ell\rightarrow\infty$.  Applying Gronwall's inequality, we find 
$$
\int^T_0\int_U|D\bu^{N_k}-D\bu^{N_\ell}|^2dxdt
\le  \frac{e^{2LT}-1}{2L}o(1)
$$
as $k,\ell\rightarrow\infty$.  The assertion follows as $L^2([0,T];H^1_0(U;\R^d))$ is complete. 
\end{proof}
At last, we will verify the existence of a weak solution of the initial value problem  \eqref{IVPmu}. 
\begin{thm}
For each $\bg\in H^1_0(U;\R^d)$ and $\bh\in L^2(U;\R^d)$, there exists a weak solution $\bu$ of \eqref{IVPmu}.
\end{thm}
\begin{proof}
Let $(\bu^N)_{N\in \N}$ denote the sequence of weak solutions we obtained from Galerkin's method. Note that for each test mapping $\phi\in C^\infty_c(U\times[0,T);\R^d)$,
\begin{align*}
\int^T_0\int_U \phi_t\cdot \bu^N_tdxdt+\int_U \phi|_{t=0}\cdot \bh^N dx=\int^T_0\int_U D\phi\cdot\left( DF(D\bu^N)+\mu D\bu^N_t\right)dxdt.
\end{align*}
By Lemma \ref{LastCompactUNlemma}, there is a subsequence $(\bu^{N_k})_{k\in \N}$ which converges 
in various senses to a mapping $\bu$ which satisfies $\bu(\cdot,0)=\bg$ and \eqref{SpacesMu}. Moreover, we can send $N=N_k\rightarrow \infty$ above to conclude that $\bu$ satisfies \eqref{WeadSolnDamped}. It follows that $\bu$ is the desired weak solution. 
\end{proof}

\appendix

\bibliography{NewtonBib}{}

\bibliographystyle{plain}

\end{document}